\documentclass[a4paper,11pt]{article}

\usepackage{amsmath}
\usepackage{amsthm}
\usepackage{amsfonts}
\usepackage{amssymb}
\usepackage{latexsym}
\usepackage{color}
\usepackage{mathrsfs}
\usepackage{enumerate}
\usepackage{ulem}

\definecolor{ForestGreen}{cmyk}{0.92,0.00,0.59,0.25}

\definecolor{PineGreen}{cmyk}{0.92,0.00,0.59,0.25}
\definecolor{ForestGreen}{cmyk}{0.91,0.00,0.88,0.12}
\definecolor{RawSienna}{cmyk}{0.00,0.72,1.00,0.45}
\definecolor{Mulbery}{cmyk}{0.34,0.90,0.00,0.02}
\definecolor{Sepia}{cmyk}{0.00,0.83,1.00,0.70}
\definecolor{Mahogany}{cmyk}{0.00,0.85,0.87,0.35}

\allowdisplaybreaks[4]

\usepackage{version}	

\newtheorem{theorem}{Theorem}[section]
\newtheorem{proposition}[theorem]{Proposition}
\newtheorem{lemma}[theorem]{Lemma}

\theoremstyle{definition}
\newtheorem{definition}[theorem]{Definition}
\newtheorem{remark}[theorem]{Remark}
\newtheorem{example}[theorem]{Example}

\def\bE{{\mathbb{E}}}

\def\bM{{\mathbb{M}}}
\def\bN{{\mathbb{N}}}

\def\bP{{\mathbb{P}}}

\def\bR{{\mathbb{R}}}

\def\bfA{{\mathbf{A}}}

\def\cB{{\mathcal{B}}}
\def\cC{{\mathcal{C}}}

\def\cE{{\mathcal{E}}}
\def\cF{{\mathcal{F}}}

\def\cL{{\mathcal{L}}}

\def\cS{{\mathcal{S}}}

\def\frA{{\mathfrak{A}}}

\def\frD{{\mathfrak{D}}}

\numberwithin{equation}{section}

\def\1{{\mathbf{1}}}

\newcommand{\form}{{\cal E}}
\newcommand{\dom}{{\cal F}}
\newcommand{\real}{{\mathbb R}}

\newcommand{\capa}{\mathrm{Cap}}
\newcommand{\ds}{\displaystyle}
\newcommand{\la}{\label}

\newcommand{\loc}{{\sf loc}}
\newcommand{\ov}{\overline}

\newcommand{\supp}{\mathrm{supp}\,}

\renewcommand{\Cap}{{\rm Cap}\,}
\newcommand{\mae}{m\text{-a.e.}}
\newcommand{\on}{\ \text{on}\ }
\newcommand{\qe}{\ {\rm q.e.}\ }

\title{On a convergence of positive continuous additive functionals in terms of their smooth measures}

\author{Yasuhito Nishimori\thanks{Department of General Education,
National Institute of Technology, Anan College, Anan, Tokushima,
774-0017, Japan
({\sf nishimori@anan-nct.ac.jp}).} \qquad 
Matsuyo Tomisaki\thanks{Department of Mathematics, Nara Women's
University, 
Kita-Uoya Nishimachi, Nara, 630-8506, Japan
({\sf tomisaki@cc.nara-wu.ac.jp}).}  \qquad  
Kaneharu Tsuchida\thanks{Department of Mathematics,
National Defense Academy,
Yokosuka, Kanagawa, 239-8686, Japan
({\sf tsuchida@nda.ac.jp}). } \\
\ and \  \ 
Toshihiro Uemura\thanks{Department of Mathematics, 
Faculty of Engineering Science, 
Kansai University, Suita, Osaka 564-8680, Japan
({\sf t-uemura@kansai-u.ac.jp}). }
}

\date{\today}

\begin{document}

\baselineskip=16pt

\maketitle
 \begin{abstract}
  A compactness of the Revuz map is established in the sense that the locally uniform convergence of a sequence of
  positive continuous additive functionals is derived in terms of their smooth measures.  
  To this end, we first introduce a metric on the space of
  measures of finite energy integrals and show some structures of the
  metric. Then, we show the compactness and  give some examples of positive
  continuous additive functionals that the convergence holds in terms of the associated smooth measures.
 \end{abstract}

 
\section{Introduction}
The positive continuous additive functionals ({\it abbrev.} {\sf PCAF}s) play an important role in the study of Markov processes  
and the Dirichlet spaces,  in particular, in the study of transformations of the Markov processes including time changes, 
 Feynman-Kac and Girsanov transformations (see {\it e.g.},  \cite{Sh, FOT, CF}).
There  is a one-to-one correspondence between the class of {\sf PCAF}s 
($\bfA_c^+$ in notation) and the class of nonnegative Borel measures charging no set of zero capacities,  called the 
{\it smooth  measures} ($\cS$ in notation). This correspondence is called the {\it Revuz correspondence}.

In the present paper  we will focus on a convergence of  {\sf PCAF}s.
In particular, we will  study the locally uniform convergence of {\sf PCAF}s in terms of some convergences 
of their smooth measures. Recently, in the case that the planer Brownian motion and the 
one-dimensional symmetric $\alpha$-stable process,  this type of the convergence has been studied 
in \cite{AK, GRV, O} to construct the Liouville Brownian motion and the Liouville $\alpha$-stable process. 
In \cite{CTU},  using smooth measures, $L^2$-convergences of Feynman-Kac semigroups defined 
through the {\sf PCAF}s are considered.

To explain a more specific  situation,  let $E$ be a locally compact separable metric space $E$ and $m$ 
a positive Radon measure on $E$ with full support.  We are given an $m$-symmetric 
Hunt process $\bM = (X_t, \bP_x)$ on $E$ and  the associated Dirichlet form $(\form, \dom)$ on 
$L^2(E;m)$ is assumed to be regular.  Then the {\sf PCAF}s and the smooth measures ${\cal S}$ 
generated by $\bM$ or $(\form, \dom)$ are defined.  
If $f$ is a nonnegative Borel function on $E$, then ${\sf A}_t = \int_{0}^{t} f(X_s) ds$ defines a 
{\sf PCAF} and the corresponding smooth measure (called the {\it Revuz measure} of $\{{\sf A}_t\}$) 
is $f(x)m(dx)$.  For the one-dimensional Brownian motion, the delta measure $\delta_a(dx)$ at a point 
$a\in \real$ is a smooth measure whose {\sf PCAF} is nothing but the local time $\{\ell(t,a)\}_{t\ge 0}$. 

In the  following, we will regard the Revuz correspondence as a map from the space $\cS$ to 
$\bfA_c^+$,  call it {\it Revuz map}, and study the convergence of {\sf PCAF}s via some convergences 
of their smooth measures from this point of view.  To the best of our knowledge, there does not seem 
to be much research on such ``topological properties of the mapping from $\cS$ to $\bfA_c^+$''. 
However, since ${\cal S}$  is quite large (larger  than the class of positive Radon measures charging 
no set of zero capacities) and so is $\bfA_c^+$,  it is  harder to put a topological structure  to this space 
$\bfA_c^+$ through ${\cal S}$ in general.  We therefore turn our attention to some restricted 
classes of smooth measures first.  One is the set of positive Radon measures which are of finite energy 
integrals ($\cS_0$ in notation). For any $\mu\in {\cal S}_0$ and $\alpha>0$, there exists a unique 
element $U_\alpha\mu \in\dom$, called the $\alpha$-potential of $\mu$, such that 
$$
\int_E \varphi(x) \mu(dx)=\form_\alpha(U_\alpha \mu ,\varphi):=\form(U_\alpha \mu, \varphi)+
\alpha (U_\alpha \mu, \varphi), \quad \varphi \in \dom \cap C_0(E),
$$
where $C_0(E)$ is the set of continuous functions on $E$ with compact support and 
$(\cdot, \cdot)$ means the $L^2$-inner product of $L^2(E;m)$. 
Using the $1$-potentials of measures $\mu$ and $\nu$ in ${\cal S}_0$, we can define a metric on 
${\cal S}_0$ as follows: 
$$
\rho(\mu, \nu):= \sqrt{\form_1\big(U_1\mu-U_1\nu, U_1\mu-U_1\nu)}, \quad \mu, \nu \in {\cal S}_0.
$$
Then we show that the space $\cS_0$ is a Polish space with the metric $\rho$ (Proposition 
\ref{prop-Polish}).  This result itself is well-known in the classical potential theory for the cases of the 
Newton potential  and the Riesz potential  (see \cite{Land}). Note that our result holds true for potential 
free situations as far as in a regular Dirichlet form setting.  We then consider the locally uniform 
convergence of {\sf PCAF}s for the $\bP_x$-almost surely (for quasi everywhere $x \in E$) in terms of 
their smooth measures in ${\cal S}_0$ through this metric $\rho$. After establishing a convergence 
result of {\sf PCAF}s via that of their smooth measures in ${\cal S}_0$ (Theorem \ref{thm-rho}), we turn 
to consider the case when the measures are in ${\cal S}$. Since a measure in ${\cal S}$ may not have 
the $1$-potential in general, we can not directly apply the method for proving the convergence in the 
case of the measures of finite energy integrals.  To overcome this difficulty, we take a proper increasing
 sequence of closed sets such that smooth measures restricted on each closed sets are of finite 
 energy integrals.  Then under some assumptions, we can obtain the convergence of {\sf PCAF}s 
 via the convergence of their smooth measures (Theorem \ref{thm-main-02}).

We next introduce a class of smooth measures that are absolutely continuous with respect to the 
basic measure $m$.  Then imposing some conditions on the Radon-Nikod\'ym densities,  
we obtain a convergence of the {\sf PCAF}s  associated with such measures in ${\cal S}$ 
(Theorem \ref{thm-02}).   In particular,  the convergence of {\sf PCAF}s holds when the sequence of 
the densities of smooth measures converges in $L^p$ for $p>1$ (Proposition \ref{prop-AC}). 
We stress that we can allow the exponents in the Lebesgue spaces for which the densities of 
smooth measures belong to may vary.

The structure of this paper is the following.  In section \ref{sec2}, we collect some basic notions and 
notations of the Dirichlet spaces and the Markov processes, especially, notions of capacity, smooth 
measures,  measures of finite energy integrals and additive functionals.  After this, we show in 
section \ref{sec3} that the space $\cS_0$ is indeed a Polish space with  the metric $\rho$ and 
the convergence of measures in ${\cal S}_0$  implies that the measures converge  {\it vaguely}. 
Conversely, if a sequence of measures  in ${\cal S}_0$ having the uniformly bounded 
$1$-potentials converges vaguely, the measures  converges {\it ``weakly in ${\cal S}_0$ with 
respect to $\rho$''} (Proposition \ref{prop-weak}).  In section \ref{sec4}, we prove the convergence 
of {\sf PCAF}s through the convergence of their measures in $\cS_0$ with respect to $\rho$ 
by using an idea in \cite{FOT}. We proceed to consider the convergence of {\sf PCAF}s through 
the convergence of their smooth measures in the class $\cS$ in section \ref{sec5}. 
In section \ref{sec6}, we assume that the measures have densities.  We introduce an approximating 
method of the {\sf PCAF} associated with $fm$ for a given density $f$.  Even though the density 
function may  not have $L^p$-integrability for all $p \in [1,\infty)$,  if the density function is locally 
integrable and it can be approximated by a sequence $f_n$ in the meaning of Definition \ref{def-A}, 
then the {\sf PCAF} of $f m$ is also approximated the {\sf PCAF}s of $f_n  m$.  This provides 
one perspective to observe the topological structure of the space of {\sf PCAF}s from the 
corresponding measure space.  In section \ref{sec7}, we first show some examples of the density 
functions concerned in section \ref{sec6}.  Then, we discuss on the {\sf PCAF}s which are determined 
by diffusion processes.


\section{Preliminaries}\label{sec2}
Let $E$ be a locally compact separable metric space and
$m$ a positive Radon measure on $E$ with full topological
support. Let $E_{\partial} = E \cup \{\partial\}$ be the
one-point compactification of $E$. When $E$ is already compact,
$\partial$ is regarded as an isolated point.  Let $\bM = (\Omega, \cF, \cF_t, X_t, \bP_x,  \zeta)$ be an 
$m$-symmetric Hunt process on $E$, where $\{\cF_t\}$ is the minimum augmented filtration and $\zeta$ 
is the life time associated with $\bM$,
that is, $\zeta(\omega) := \inf\{t \ge 0 : X_t = \partial\}$.

Let $\{p_t\}_{t \ge 0}$ be the transition semigroup of $\bM$, that is, $p_t f(x) = \bE_x[f(X_t)]$.
By \cite[Lemma 1.4.3]{FOT}, we know that $\{p_t\}_{t \ge 0}$ uniquely determines
a strongly continuous Markovian semigroup $\{T_t\}_{t \ge 0}$ on $L^2(E;m)$. 
Then the Dirichlet form on $L^2(E;m)$ of $\bM$ is defined by:
$$
\left\{
\begin{array}{rl}
  \ds \cE(u,v)  \!\!  & = \ds \lim_{t \downarrow 0} \frac{1}{t} (u - T_t u, v), \\
 & \vspace*{-6pt} \\
  \ds \cF \!\! & \ds = \left\{u \in L^2(E;m) :   \lim_{t \downarrow 0} \frac{1}{t}(u - T_t u, u) < \infty\right\}, \\
\end{array}
\right.
$$
 where $(f,g)$ (resp. $\|f\|_2$) means the $L^2$-inner product of $f,g$  (resp. the $L^2$-norm of $f$).  
 For $\alpha > 0$ and $u,v\in \dom$,  denote by $\form_\alpha(u,v):= \cE(u,v) + \alpha (u,v)$.  
 Throughout this paper, we assume that $(\cE,\cF)$ is {\it regular}  in the sense that 
 $\dom\cap C_0(E)$ is dense in $\dom$ with respect to $\sqrt{\form_1}$-norm and is also dense in 
 $C_0(E)$ with respect to sup-norm $\|\cdot\|_\infty$.
 %

Now we  define the ($1$-)capacity associated with the Dirichlet form $(\cE,\cF)$. 
For an open set $O\subset E$, 
$$
\Cap(O):=
 \begin{cases}
  \inf \big\{\cE_1(u,u) : u \in \cL_O \big\},  &  {\rm if} \ \ \cL_O \not= \varnothing, \\
  \infty,  &  {\rm if} \ \ \cL_O = \varnothing, 
 \end{cases}
$$
where $\cL_O := \{u \in \cF : u \ge 1\ \mae \on O\}$,  and  for any set $A \subset E$,
 \[
  \Cap(A) := \inf\big\{\Cap(O) : A \subset O,\ \text{$O$ is open} \big\}.
  \]
Then it is known that $\Cap$ is a Choquet capacity (see \cite{FOT}).  
We call a  set $A \subset E$ {\it exceptional} if $\Cap(A)=0$.  
A statement depending on $x \in A$ is said to hold q.e. on $A$ if there exists an 
exceptional set $N \subset A$ such that the statement is true for every $x \in A \setminus N$.
  Let $u$ be a numerical function defined q.e. on $E$.  The function $u$ is said to be {\it quasi-continuous} 
  if for any $\varepsilon > 0$, there exists an open set $G \subset E$ such that $\Cap(G) < \varepsilon$ and
  $u|_{E \setminus G}$ is finite and continuous.
  Here $u|_{E \setminus G}$ means the restriction of $u$ to
  the set $E \setminus G$.
  We say an increasing sequence $\{F_n\}$ of closed sets is
  a {\it nest} if $\Cap(E \setminus F_n)$ decreases to $0$ as $n \to \infty$.

  We now give the definition of smooth measures which play an important role in this paper.  
   A positive Borel measure $\mu$ on $E$ is called {\it smooth} if the following two conditions 
   are satisfied:
   \begin{description}
    \item[(S.1)] $\mu$ charges no set of zero capacity,
    \item[(S.2)] there exists an increasing sequence
	       $\{F_n\}$ of closed sets such that
	       \[
	       \mu(F_n) < \infty\ (n = 1,2, \cdots)
	       \quad \text{and} \quad
	       \lim_{n \to \infty} \Cap(K \setminus F_n) = 0
	       \quad \text{for any compact set $K$}.
	       \]
   \end{description}
   We denote the set of all smooth measures by $\cS$.
   An increasing sequence $\{F_n\}$ of closed sets
   satisfying the condition of the last part in {\bf (S.2)} is called  a {\it generalized nest}. 
A positive Radon measure $\mu$ on $E$ is said to be of 	{\it finite energy integrals} 
if there exists a positive constant $C>0$ such that 
	  \[
	  \int_{E} |v(x)| \mu(dx) \le C \sqrt{\cE_1(v,v)},
	  \quad v \in \cF \cap C_0(E).
	  \]
The set of all measures on $E$ which are of finite energy integrals is denoted by ${\cal S}_0$.
Then, according to the Riesz representation theorem, we find that  
  $\mu \in{\cal S}_0$ if and only if  there exists a unique element
  $U_{\alpha} \mu \in \cF$, called 
  the {\it  $\alpha$-potential} of $\mu$,  for each $\alpha > 0$ such that
   \[
    \cE_{\alpha}(U_{\alpha}\mu, v) =
    \int_{E} v(x) \mu(dx), \quad \text{for any }
    v \in \cF \cap C_0(E).
    \]
A subclass $\cS_{00}$ of $\cS_0$ is defined as follows:
	  \[
	  \cS_{00} = \big\{\mu \in S_0 :
	  \mu(E) < \infty,\ \|U_1 \mu\|_{\infty} < \infty \big\}, 
	  \]
	  where $\| \cdot \|_{\infty}$ is the $L^{\infty}$-norm with
	  respect to the measure $m$.   
    The following theorem shows some characterizations of the exceptional sets. 

    \begin{theorem}[{\bf \cite[Theorem 2.2.3]{FOT}}]
     \la{thm-equi}
     The following conditions are equivalent for a Borel set $B \subset E:$
     \begin{enumerate}[{\rm (i)}]
      \item $\Cap(B) = 0$,
      \item $\mu(B) = 0 $, for any $\mu \in \cS_0$,
      \item $\mu(B) = 0$, for any $\mu \in \cS_{00}$.
     \end{enumerate}
    \end{theorem}

   We now introduce the definition of a continuous additive functional.

    \begin{definition}
     \la{def-AF}
     An extended real valued
     (or numerical) %
     function ${\sf A}_t(\omega), \ t \ge 0,\ \omega \in \Omega$,  is 
      a {\it continuous additive functional} ({\sf CAF} for short) if the following conditions must hold:
     \begin{description}
      \item[(A.1)] for each $t \ge 0$, ${\sf A}_t(\cdot)$ is $\cF_t$-measurable,
      \item[(A.2)] there exist a set $\Lambda \in \cF_{\infty}$ and an exceptional set
		 $N \subset E$ such that $\bP_x(\Lambda) = 1$ for any
		 $x \in E \setminus N$ and $\theta_t \Lambda \subset \Lambda$, 
      \item[(A.3)] for each $\omega \in \Lambda$, the map $t\mapsto {\sf A}_t(\omega)$ is continuous
		 on $[0, \infty)$.
     \end{description}
     When a {\sf CAF} ${\sf A} = \{{\sf A}_t\}$ only takes values in $[0,\infty]$, we call ${\sf A}$ a
     {\it positive continuous additive functional} ({\sf PCAF} for short).
     The set of all {\sf CAF}s (resp. {\sf PCAF}s) is denoted
     by $\bfA_c$ (resp. $\bfA_c^+$). 
    \end{definition}

    Two {\sf CAF}s ${\sf A}^{(1)}$ and ${\sf A}^{(2)}$ are said to be {\it equivalent} if for each $t > 0$,
    $\bP_x({\sf A}_t^{(1)} = {\sf A}_t^{(2)}) = 1\ \qe x \in E$.
    If ${\sf A}^{(1)}$ and ${\sf A}^{(2)}$ are equivalent,
    we write ${\sf A}^{(1)} \sim {\sf A}^{(2)}$.
    It is clear that this binomial relation ``$\sim$'' defines an equivalence relation on $\bfA_c$.
    
    The next theorem characterizes a concrete relation between measures in $\cS$ and {\sf PCAF}s 
    (in $\bfA_c^+$ under this equivalence).

    \begin{theorem}[{\bf \cite[Theorem 5.1.4]{FOT}}]
     \la{thm-Revuz}
     The quotient space of $\bfA_c^+$ under the equivalence relation ``$\sim$'' and the family $\cS$
     are in one-to-one correspondence under the following relation:
     For $\mu \in \cS$ and ${\sf A} \in \bfA_c^+$, 
     \begin{equation}
      \lim_{t \downarrow 0} \frac{1}{t} \int_{E} \bE_x\left[\int_{0}^{t} f(X_s) d{\sf A}_s\right]m(dx)
       = \int_{E} f(x) h(x) \mu(dx)
       \la{eq-Revuz}
     \end{equation}
     where $h$ is any $\gamma$-excessive fiction $(\gamma \ge 0)$ and $f \in \cB_+(E)$. 
    \end{theorem}

    The equation \eqref{eq-Revuz} is called the {\it Revuz
    correspondence}.
    In this paper, we call it the {\it Revuz map}
    by viewing this correspondence from $\cS$ to $\bfA_c^+$.

\section{Some properties of measures of finite energy integrals}\label{sec3}

    In this section, we consider some (metric) properties  of  the space of measures
    which are of finite energy integrals, ${\cal S}_0$. 
    As is mentioned in the preceding section, any measure $\mu$ in $\cS_0$ has
    the $1$-potential $U_1 \mu \in\dom$. So, using these $1$-potentials of measures we define a metric 
    on $\cS_0$.  To this end, let us denote $\| \cdot \|_{\cE_1}$  the norm on $\dom$ with respect to 
    $\cE_1$, that is,
    $$
    \| f \|_{\cE_1} := \sqrt{ \cE_1 (f,f)} =\sqrt{\form(f,f)+(f,f)}, \quad  f \in \cF.
    $$
Then we define a  function  $\rho : \cS_0 \times \cS_0 \to [0,\infty)$ by
    \begin{equation}
     \rho(\mu,\nu) := \|U_1 \mu - U_1 \nu \|_{\cE_1}, \quad \mu, \nu \in {\cal S}_0.
      \label{eq:def-rho}
    \end{equation}
We show that the function $\rho$ becomes a metric on $\cS_0$ and has  nicer properties.
    \begin{lemma}
     \la{lem-metric}
     The function $\rho$ defines a metric on $\cS_0$.
    \end{lemma}
    \begin{proof}
     For $\mu, \nu \in \cS_0$,
     it is clear that $\rho(\mu,\nu) \ge 0$ and 
     it holds that
     \[
     \rho(\mu, \nu) =
     \|U_1 \mu - U_1 \nu\|_{\cE_1}=
     \|U_1 \nu - U_1 \mu\|_{\cE_1}
     = \rho(\nu,\mu).
     \]
     For any $\xi \in \cS_0$, 
     \[
     \rho(\mu,\nu) = \|U_1 \mu - U_1 \nu\|_{\cE_1}
     \le \|U_1 \mu - U_1 \xi\|_{\cE_1} +
     \|U_1\xi - U_1 \nu \|_{\cE_1} = \rho(\mu,\xi) + \rho(\xi, \nu). 
     \]
     Hence we obtain the triangle inequality with respect to $\rho$.
     Finally, assume that $\rho(\mu,\nu) = 0$. Then $U_1 \mu = U_1\nu$. So it follows that  
     \[
     \int_{E} \varphi(x) \mu(dx) = \cE_1(U_1 \mu, \varphi)
     = \cE_1(U_1 \nu, \varphi) = \int_{E} \varphi(x) \nu(dx),\quad
     \text{for any $\varphi \in \cF \cap C_0(E)$}.
     \]
     This implies that $\mu = \nu$ because $\mu, \nu$ are
     Radon measures and the Dirichlet form $(\form, \dom)$ is regular. 
     Therefore $\rho$ defines a metric on $\cS_0$.
    \end{proof}
    \begin{remark}    \la{rem-rho-metric}
     For $\alpha > 0$, if we define a function $\rho_\alpha$ on ${\cal S}_0\times {\cal S}_0$ by 
     $$
     \rho_\alpha(\mu, \nu):=\| U_\alpha \mu -U_\alpha \nu\|_{\form_\alpha}
     :=\sqrt{\form_\alpha (U_\alpha\mu-U_\alpha\nu, U_\alpha\mu-U_\alpha\nu)}, 
     \quad \mu, \nu \in {\cal S}_0,
     $$
	we find that  
     $$
     \sqrt{1/(\alpha \vee 1)} \, \rho(\mu, \nu) \le  \rho_\alpha(\mu, \nu) \le \sqrt{(1/\alpha) \vee 1} 
     \, \rho(\mu, \nu), \quad \mu, \nu \in{\cal S}_0
     $$
    holds.  Then it  follows that $\rho_{\alpha}$ is also a metric on ${\cal S}_0$ and it is equivalent 
    to $\rho$ for each $\alpha>0$.  
    \end{remark}

\medskip
    Put $f = U_1 \mu$ for $\mu \in \cS_0$ and define for $n \in \bN$
    \begin{equation}
     g_n(x) = n(f(x) - nR_{n+1}f(x)), 
      \la{eq:approx-g}
    \end{equation}
    where $R_n$ denotes the $n$-resolvent associated with the process
    $\bM$. Since $f$ is $1$-excessive, we know that
    $g_n \ge 0\ \mae$ by \cite[Theorem 2.2.1]{FOT}.
    
    \begin{lemma}
     \la{lem-strong}
     Let $\mu$ be a measure in $\cS_0$ and put $f = U_1 \mu$. Then
     $R_1 g_n$ converges to $f$ in $\cE_1$-strongly as $n \to \infty$.
    \end{lemma}
    \begin{proof}
     By the resolvent equation, the following identities hold:
	\begin{align*}
      R_1 g_n(x) &= n(R_1 f(x) - n R_{n+1} R_1 f(x)) \\
      &= n(R_1 f(x) - R_1 f(x) + R_{n+1}f(x)) \\
      &= nR_{n+1} f(x).
	\end{align*}
     Since $f$ is a $1$-excessive function (cf. \cite[Theorem 2.2.1]{FOT}), we know that
     \[
      R_1 g_n = n R_{n+1} f \uparrow f \quad (n \to \infty).
     \]
     Then it holds that 
	\begin{align*}
      \cE_1(R_1 g_n - f, R_1 g_n - f)
      &= \cE_1(R_1 g_n, R_1 g_n) - 2\cE_1(R_1 g_n, f) + \cE_1(f,f) \\
      &= (g_n, R_1 g_n) - 2\cE_1(R_1g_n, f) + \cE_1(f,f) \\
      &\le (g_n, f) - 2\cE_1(R_1 g_n, f) + \cE_1(f,f) \\
      &= - \cE_1(R_1 g_n, f) + \cE_1(f,f)  \ \to  \ 0 \quad (n \to \infty), 
	\end{align*}
     because $g_n$ is non-negative and
     $R_1 g_n \to f$ in
     $\cE_1$-weakly by \cite[Lemma 2.2.2]{FOT}.
    \end{proof}

    \begin{lemma}
     \la{lem-sep}
     The metric space $(\cS_0,\rho)$ is separable. 
    \end{lemma}
    \begin{proof}
     Since $E$ is a locally compact separable metric space and
     $m$ is a Radon measure on $E$,
     we see that %
     the space $L^2(E;m)$ is separable.
     Let $\frD$ be a countable and dense
     subset of $L^2(E;m)$. For a measure $\mu \in \cS_0$,
     let $f$ be the $1$-potential of $\mu$, that is, $f = U_1 \mu$.
     Let $g_n$ be the function defined in \eqref{eq:approx-g}. 
     Then we know that $g_n \in \cF \subset L^2(E;m)$ with $g_n\ge 0$ for all $n \in \bN$.
     By Lemma \ref{lem-strong}, the following holds, 
     \[
      R_1 g_n \to f, \quad \cE_1\text{-strongly}.
     \]
     Hence we know that for any $\varepsilon > 0$, there exists a number
     $N \in \bN$ such that
     $\rho(g_n m, \mu) < \varepsilon/2$ for $n \ge N$. 
     Since $g_n$ is in $L^2(E;m)$, there exists a function $h_n \in \frD$ with $h_n\ge 0$ 
     such that $\|g_n - h_n\|_2 < \varepsilon/2$.
     Then noting that 
     \begin{align*}
      \rho(g_n m, h_n m)^2 
      &=
       \cE_1(R_1 g_n - R_1 h_n, R_1 g_n - R_1 h_n) 
       = \cE_1(R_1(g_n - h_n), R_1(g_n - h_n)) \\%
      &= \int_{E} (g_n(x) - h_n(x)) R_1(g_n - h_n)(x) m(dx) 
      \le \|g_n - h_n\|_2 \|R_1(g_n - h_n)\|_2 \\%
      &\le \|g_n - h_n\|_2^2 < \left(\frac{\varepsilon}{2}\right)^2,
     \end{align*}
     we obtain that $\rho(g_n m, h_n m) < \varepsilon/2$.
     If $n \ge N$, it holds that 
     \[
     \rho(\mu, h_n m) \le \rho(\mu,g_n m) + \rho(g_n m, h_n m)
     < \frac{\varepsilon}{2} + \frac{\varepsilon}{2} = \varepsilon.
     \]
     If we put $\cS_a$ by the set
     $\{\mu \in \cS_0 : \mu = f  m,\ f \in \frD_+\}$,
     $\cS_a$ is countable and dense in $\cS_0$.
     Therefore the separability of $(\cS_0,\rho)$ is proved.
    \end{proof}

     \begin{lemma}
      \la{lem-comp}
     The metric space $(\cS_0,\rho)$ is complete.
     \end{lemma}
    \begin{proof}
     Let $\{\mu_n\} \subset \cS_0$ be a $\rho$-Cauchy sequence.
     Noting that 
     \[
      \rho(\mu_m, \mu_n) = \|U_1 \mu_m - U_1 \mu_n\|_{\cE_1}
     \]
     and $(\cF,\cE_1)$ is a Hilbert space, we know that
     there exists a function $h \in \cF$ such that
     \[
      \lim_{n \to \infty} \|U_1 \mu_n - h\|_{\cE_1} = 0.
     \]
     Hence we know %
     that for any $v \in \cF \cap C_0(E)$ with $v \ge 0$,
     \begin{equation}
      \lim_{n \to \infty} \cE_1(U_1 \mu_n, v) = \cE_1(h, v).
       \label{eq:rho-01}
     \end{equation}
     By \cite[Theorem 2.2.1]{FOT}, the left hand side of \eqref{eq:rho-01}
     is non-negative, so is the right hand side.
     Using \cite[Theorem 2.2.1]{FOT} again, we can see that
     $h$ is a $1$-potential, hence there exists $\nu \in \cS_0$
     such that $h = U_1 \nu$. This implies that
     \[
     \lim_{n \to \infty} \rho(\mu_n, \nu) = \lim_{n \to \infty}
     \|U_1 \mu_n - U_1 \nu\|_{\cE_1} = 0.
     \]
     This proof is completed.
    \end{proof}

    By Lemmas \ref{lem-sep} and \ref{lem-comp}, we obtain
    the following proposition.

     \begin{proposition}
      \la{prop-Polish}
      The space $\cS_0$ is a Polish space. 
     \end{proposition}

\begin{example}
\label{ex-density}
Let $\mu = f m$ and $\mu_n = f _n m$ for $n\in {\bN}$, where $f, \ f_n \in L^2(E;m) $ and 
$\| f- f_n \|_2 \to 0$ as $n \to \infty$. Then $ \mu,\ \mu_n \in {\cS}_0$ and 
$\rho( \mu, \mu_n) =\| R_1f -R_1 f_n \|_{\cE_1} \leq \| f - f_n \|_2$. 
Hence $ \rho( \mu , \mu_n) \to 0$ as $ n \to \infty$.
\end{example}

    In the rest of this section, we collect some convergence results of the measures of $\cS_0$  
    related to $\rho$.

    \begin{proposition}
     \label{prop-01}
     Assume that $\{\mu_n\}$ are measures in $\cS_0$ which converges to
     some $\mu \in \cS_0$ in $\rho$. Then it follows that
     $\{\mu_n\}$ converges to $\mu$ vaguely, that is, for any
     $\varphi \in C_0(E)$,
     \[
      \lim_{n \to \infty} \int_{E} \varphi(x) \mu_n(dx) = \int_{E}
     \varphi(x) \mu(dx).
     \]
    \end{proposition}
    \begin{proof}
     Suppose that
     $\lim_{n \to \infty} \rho(\mu_n, \mu) =
     \|U_1 \mu_n - U_1  \mu \|_{\cE_1}
     = 0.$ Then the sequence $\big\{ \| U_1 \mu_n  \|_{\cE_1}  \big\}_{n\ge 1}$
     is bounded in $\bR$. Moreover, for any $\varphi \in \cF \cap
     C_0(E)$,
     \begin{equation}
       \int_{E} \varphi(x) \mu_n(dx) = \cE_1(\varphi, U_1 \mu_n)
     \xrightarrow{n \to \infty} \cE_1(\varphi, U_1 \mu) = \int_{E}
     \varphi(x) \mu(dx).
     \label{eq:prop-01-1}
     \end{equation}
     By the regularity of the Dirichlet form $(\cE, \cF)$, we find that
     for any $\varphi \in C_0(E)$, there exists a sequence
     $\{\varphi_{\ell}\} \subset \cF \cap C_0(E)$ such that
     \[
      \lim_{\ell \to \infty} \|\varphi - \varphi_{\ell}\|_{\infty} = 0.
     \]
     For any compact set $K$ and a relatively compact open set $G$ with
     $\supp {\varphi} \subset K^o \subset K \subset G$,
     we can take a function $\eta \in \cF \cap C_0(E)$ such that
     \[
      \eta = 1\ \text{on}\ K,\ \eta = 0\ \text{on}\ G^c\ \text{and}\ 0
     \le \eta \le 1\ \text{on}\ E.
     \]
     Note that $\varphi_{\ell} \eta \in \cF \cap C_0(E)$ and $\varphi =
     \varphi \eta$. Then we find that
	\begin{align}
      \left|\int_{E} \varphi(x) \mu(dx) - \int_{E} \varphi(x)
       \mu_n(dx)\right| 
       &= \left|\int_{E} \varphi(x) \eta(x) \mu(dx) -
			    \int_{E} \varphi(x) \eta(x)
			    \mu_n(dx)\right| \nonumber \\
      &\le
      \left|\int_{E} \varphi(x) \eta(x) \mu(dx) - \int_{E}
	     \varphi_{\ell}(x) \eta(x) \mu(dx)\right| \nonumber \\
      &  \quad + \left|\int_{E} \varphi_{\ell}(x) \eta(x) \mu(dx) - \int_{E}
	      \varphi_{\ell}(x) \eta(x) \mu_n(dx)\right| \nonumber \\
      &  \quad + \left|\int_{E} \varphi_{\ell}(x) \eta(x) \mu_n(dx) -
	      \int_{E} \varphi(x) \eta(x) \mu_n(dx)\right| \nonumber \\
      &=: 
      {\rm (I)} + {\rm (II)} + {\rm (III)}. 
      \label{eq:prop-01-2}
      \end{align}
     By the definition of $\{\varphi_{\ell}\}$, 
     \begin{align*}
      {\rm (I)} &\le \int_{E} |\varphi(x) - \varphi_{\ell}(x)| \eta(x)
       \mu(dx) 
      \le \|\varphi - \varphi_{\ell}\|_{\infty} \int_{E} \eta(x)
       \mu(dx)\ \to\ 0 \quad (\ell \to \infty). %
          \end{align*}
     We see from \eqref{eq:prop-01-1} that %
     for each $\ell$, 
     \begin{align*}
      {\rm (II)} &= \left|\int_{E} \varphi_{\ell}(x)\eta(x) \mu(dx)-
		      \int_{E} \varphi_{\ell}(x) \eta(x)
		      \mu_n(dx)\right| 
      = \left|\cE_1(\varphi_{\ell}\eta,U_1\mu) -
	   \cE_1(\varphi_{\ell}\eta, U_1 \mu_n)\right|\ \to\ 0 \quad
      (n \to \infty). %
     \end{align*}
     Finally, 
     let us %
     consider ${\rm (III)}$. By the Schwarz inequality
     \begin{align*}
      {\rm (III)} &\le \int_{E} |\varphi_{\ell}(x) - \varphi(x)|
       \eta(x) \mu_n(dx)
      \le\ \|\varphi_{\ell} - \varphi\|_{\infty} \int_{E} \eta(x)
       \mu_n(dx) \\
      &= \|\varphi_{\ell} - \varphi\|_{\infty} \cE_1(\eta, U_1\mu_n)
      \le\ \|\varphi_{\ell} - \varphi\|_{\infty}
      \| \eta  \|_{\cE_1}  \| U_1 \mu_n  \|_{\cE_1}. 
     \end{align*}
     Since 
     $\{ \| U_1 \mu_n  \|_{\cE_1} \}_{n=1}^{\infty}$
     is bounded,
     we know that (III) tends to $0$ as
     $n \to \infty$ and $\ell \to \infty$.
     Therefore, taking the limit as $n \to \infty$ and $\ell \to \infty$
     in \eqref{eq:prop-01-2}, we obtain this proposition.
    \end{proof}

\bigskip
    A sequence of measures $\{\mu_n\}$ in $\cS_0$ is said to converge to
    a measure $\mu \in \cS_0$ {\it weakly with respect to $\rho$} if 
    \[
     \lim_{n \to \infty} \cE_1(U_1 \mu_n, \varphi) = \cE_1(U_1 \mu, \varphi),
     \quad \text{for any }\varphi \in \cF.
    \]

    \begin{proposition}
     \la{prop-weak}
     Let $\{\mu_n\}$ be a sequence in $\cS_0$. Assume that $\{\mu_n\}$ is bounded
     with respect to $\rho$, that is, $\sup_{n} \cE_1(U_1 \mu_n, U_1 \mu_n) < \infty$. 
     If the measures $\{\mu_n\}$ converges to a measure $\mu$ vaguely, 
     then $\mu \in \cS_0$ and $\{\mu_n\}$ converges to $\mu$ 
     weakly with respect to $\rho$. 
    \end{proposition}
    \begin{proof}
     Since  the $1$-potentials $\{U_1 \mu_n\}$ is bounded in $\dom$ with respect to $\|\cdot\|_{\form_1}$. 
    there exist a subsequence $\{\mu_{n_k}\}$ of $\{\mu_n\}$ and a function $v \in \cF$ such that
     \[
     \lim_{k \to \infty} \cE_{1}(U_1 \mu_{n_k}, \varphi) = \cE_1(v,\varphi), \quad
     \varphi \in \cF.
     \]
     by the Banach-Alaoglu theorem.  According to the assumption that $\{\mu_n\}$ converging to $\mu$ vaguely,
     \[
     \lim_{k \to \infty} \cE_1 (U_1 \mu_{n_k}, \varphi) =
     \lim_{k \to \infty} \int_{E} \varphi \, d\mu_{n_k} = \int_{E} \varphi \, d\mu, \quad
     \varphi \in \cF \cap C_0(E).
     \]
     Hence
     \[
      \cE_1(v,\varphi) = \int_{E} \varphi \, d\mu, \quad \varphi \in \cF \cap C_0(E).
     \]
     This shows $\mu \in \cS_0$ and $v = U_1 \mu$. Thus $v$ does not depend on
     subsequences, and
     \[
     \lim_{n \to \infty} \cE_1(U_1 \mu_n, \varphi) = \cE_1(U_1 \mu, \varphi),
     \quad \varphi \in \cF \cap C_0(E).
     \]
     Since $(\cE, \cF)$ is regular, for any $\psi \in \cF$ and for any $\varepsilon > 0$,
     there exists a function $\varphi \in \cF \cap C_0(E)$ such that
     $\|\psi -\varphi \|_{\cE_1} < \varepsilon$.
     Therefore we have
     \begin{align*}
      |\cE_1(U_1 \mu_n, \psi) - \cE_1(U_1 \mu, \psi)|
       &= \left|\int_{E} \tilde{\psi} \, d\mu_n - \int_{E} \tilde{\psi} \, d\mu\right| \\
      &\le \left|\int_{E} (\tilde{\psi} - \varphi) \, d\mu_n \right|
       + \left|\int_{E} \varphi \, d\mu_n - \int_{E} \varphi \, d\mu\right|
       + \left|\int_{E} (\varphi - \psi) \, d\mu\right| \\
      &= |\cE_1(U_1 \mu_n, \psi- \varphi)|
       + \left|\int_{E} \varphi \, d\mu_n - \int_{E} \varphi \, d\mu\right|
       + |\cE_1(U_1 \mu, \varphi - \psi)| \\
      &\le M \| \psi - \varphi \|_{\cE_1}
 		+ \left|\int_{E} \varphi \, d\mu_n - \int_{E} \varphi \, d\mu\right| \\
      &\le M \varepsilon + \left|\int_{E} \varphi \, d\mu_n - \int_{E} \varphi \, d\mu\right|
       \xrightarrow{n \to \infty}\ M \varepsilon
      \xrightarrow{\varepsilon \to 0}\ 0,
     \end{align*}
     where $M := \sup_n \| U_1 \mu_n \|_{\cE_1} +\| U_1 \mu \|_{\cE_1}$. 
    \end{proof}
    
        \begin{remark}
     \la{rem-Polish-signed} In the introduction, we noted that 
     Proposition \ref{prop-Polish} holds for the Newton potential or the Riesz potential in the classical 
     potential theory.  The results stated in this section moreover hold true to any potential kernels
     associated with regular Dirichlet forms, including,  of course, the cases of
     the Newton potential (corresponding to the Brownian motion) or the Riesz potential (corresponding 
     to the symmetric $\alpha$-stable process). %
     But, for the space of signed measures $\cS_0 - \cS_0$, it is not 
     complete in general (see \cite[Theorem 1.19]{Land}). 
    \end{remark}

    \section{Compactness of Revuz maps in $\cS_0$}\label{sec4}

    In this section, we show that the Revuz map 
    is compact in the sense that, for any bounded sequence in $\cS_0$ with the metric $\rho$, 
    the corresponding sequence of $\bfA_c^+$ by the Revuz map has a subsequence so that 
    the subsequence converges locally uniformly (on $[0,\infty)$).    
        Let $\rho$ be the metric on $\cS_0$ defined by \eqref{eq:def-rho}.  

    \begin{theorem}
     \la{thm-rho}
Let  ${\sf A}^n, \, {\sf A} \in \bfA_c^+$.  Denote by $\mu_n \,  \mu$ their Revuz measures 
in ${\cal S}$ for $n\ge1$. Assume that all $\mu_n$ and $\mu$ belong to ${\cal S}_0$.  
If $\lim_{n \to \infty}\rho(\mu_n, \mu) = 0$,  
     then there exists a subsequence $\{n_k\}$ of $\{n\}$ such that 
\[
	\bP_x \left( 
		\lim_{n_k \to \infty} 
			{\sf A}^{n_k}_t ={\sf A}_t \text{ locally uniformly in $t$ on} \  [0,\infty) 
	\right) = 1, \quad \qe x \in E.
\]
    \end{theorem}

The proof of Theorem \ref{thm-rho} is long. 
We divide the proof into Lemmas \ref{lem-pre-00}--\ref{lem-pre-new}. 
The next lemma reveals a duality relation between the resolvent and the potential.

      \begin{lemma}[{see also \cite[(4.1.11)]{CF}}]
      \la{lem-pre-00}
      For any $\nu \in \cS_{00},\ \alpha > 0$ and $f \in L^1(E;m)\cap{\cal B}(E)$, 
      \begin{equation}\label{newL1}
       {\bE}_{\nu}\left[ \int_0^\infty e^{-\alpha s} f(X_s) ds    \right]=
       \int_E R_\alpha f(x) \nu(dx) =\int_E f(x) U_\alpha \nu(x) m(dx).
      \end{equation}
      \end{lemma}
     \begin{proof}
      Since $\left| \int_E f(x) U_\alpha \nu(x) m(dx) \right| \leq \| U_\alpha \nu \|_\infty
      \| f\|_1 <\infty$,
      it is enough to show (\ref{newL1}) for non-negative $f\in L^1(E;m)\cap {\cal B}_+(E)$.
      Let $\{ F_n\}$ be an increasing sequence of compact sets satisfying 
      $\cup_{n=1}^\infty F_n =E$.
      Put $ f_N(x) := (f(x) \land N) \cdot 1_{F_N}(x),\ x \in E$ for each $N \in {\bN}$.
      Then it is easy to see that 
      $\{ f_N \} \subset L^2(E;m)\cap {\cal B}_+(E)$ and $(0 \leq) f_N(x) \uparrow f(x),\ x \in E $ 
      as $N \to \infty$.
      For any $N$, we see that
      \begin{align*}
       {\bE}_{\nu}\left[ \int_0^\infty e^{-\alpha s} f_N(X_s) ds    \right]=
       \int_E R_\alpha f_N(x) \nu(dx) =
       {\cE}_\alpha (R_\alpha f_N, U_\alpha \nu)=
       \int_E f_N(x)U_\alpha  \nu(x) m(dx).
      \end{align*}
      Letting $N \to \infty$,
      we get (\ref{newL1})  by the monotone convergence theorem. 
     \end{proof}

\medskip
   Let $f \in L^2(E;m)\cap {\cal B}_+(E)$ and put
     \begin{equation}\la{eq-AF1}
      {\sf A}^f(t) := \int_{0}^{t}e^{-s} f(X_s) ds.
     \end{equation}
  Note that the measure $fm$ belongs to ${\cal S}_0$ and the $1$-potential $U_1(fm)$ is equal to 
  $R_1f$, the $1$ resolvent of $f$.

   \begin{lemma}
     \la{lem-pre-01}
    For any  $\nu \in \cS_{00}$,  there exists some constant
     $M_{\nu} > 0$ such that 
     \[
     \bE_{\nu}\left[\left({\sf A}^f(\infty) - {\sf A}^g(\infty)\right)^2\right]
     \le M_{\nu} 
 \max \{ \| R_1 f \|_{\cE_1},\ \| R_1 g \|_{\cE_1} \}
 \|R_1 f - R_1 g\|_{\cE_1}, 
     \]
 for  $ f, g \in L^2(E;m) \cap {\cB}_+(E)$.
    \end{lemma}

    \begin{proof}
 For any $f, g \in L^2(E;m) \cap {\cB}_+(E)$, put $h := f - g$. 
By the Markov property, 
\begin{align*}
	\bE_{\nu}\left[\left({\sf A}^f(\infty) - {\sf A}^g(\infty)\right)^2\right]
	&=
	\bE_{\nu}\left[\left(\int_{0}^{\infty} e^{-s}
		h(X_s) ds\right)^2\right] \\
	&= 
	2 \bE_{\nu}\left[\int_{0}^{\infty} e^{-s} h(X_s)
		\left(\int_{s}^{\infty} e^{-u} h(X_u) du\right) ds\right] \\
	&= 
	2\bE_{\nu}\left[\int_{0}^{\infty} e^{-s} h(X_s)
	\left(\int_{0}^{\infty} e^{-(u+s)} h(X_{u+s})du\right) ds\right] \\
	&= 
	2 \bE_{\nu}\left[\int_{0}^{\infty}
	e^{-2s} h(X_s) R_1 h(X_s) ds\right] , 
\end{align*}
     by means of (\ref{newL1}) with $f$ replaced by $ h \cdot R_1 h \in L^1(E,m)$,
     then 
     the above is continue to 
\begin{align*}
	&= 
	2 \int_{E} R_2 (h \cdot R_1h)(x) \nu(dx) 
	=
	2 \int_{E} h (x) \cdot R_1 h(x) \cdot U_2 \nu(x) m(dx) 
	=
	 2 \cE_1 (R_1 h , R_1 ( h U_2 \nu)) 
	\\
	& \le 
	2 \| R_1 h \|_{\cE_1} \| R_1 ( h U_2 \nu) \|_{\cE_1}
	\le 
	2 \| R_1 h \|_{\cE_1} \big\{ \| R_1 ( f U_2 \nu) \|_{\cE_1}
		+ \| R_1 ( g U_2 \nu) \|_{\cE_1} \big\}.
\end{align*}
Here, since $ f \geq 0$ and $ U_2 \nu \geq 0$, 
\[
\| R_1 (f U_2 \nu) \|^2_{\cE_1} = \big( f U_2 \nu, R_1( f U_2 \nu) \big) 
\leq \| U_2 \nu  \|_\infty^2 (f, R_1 f) =\| U_2 \nu \|^2_\infty \cE_1 (R_1f, R_1 f).
\]
Therefore 
we obtain
     \[
      \bE_{\nu}\left[\left({\sf A}^f(\infty) - {\sf A}^g(\infty)\right)^2\right]
      \le 
 4  \| U_2 \nu \|_\infty   \| R_1 h \|_{\cE_1}  \max \{  \| R_1  f \|_{\cE_1} , \|  R_1 g  \|_{\cE_1} \}.
     \]
 Putting   $M_{\nu} = 4\|U_2 \nu\|_{\infty}$, we get the lemma. 
    \end{proof}

    \begin{lemma}
     \la{lem-pre-02}
 Let $ f \in L^2(E;m)\cap {\cal B}_+(E)$. Denote by ${\sf A}^{f}(t)$ the {\sf CAF} defined by 
 \eqref{eq-AF1} and 
      put 
     \begin{equation}\label{newmar}
     {\sf M}^f(t) := {\sf A}^{f}(t) + e^{-t} R_1 f (X_t). 
     \end{equation}
     Then $ \{{\sf M}^f(t)\}_{t \ge 0}$ is an $(\cF_t, \bP_{\nu})$-martingale,
 where
     $\nu \in \cS_{00}$ with $\nu(E) = 1$.
    \end{lemma}
    \begin{proof}
     For $t >s \ge 0$, 
   	\begin{align*}
     \bE_{\nu} \left[ {\sf M}^f(t) \mid \cF_s\right] 
     &= \bE_{\nu}\left[ \left. {\sf A}^{f}(t)
	      + e^{-t} R_1 f(X_t) \, \right| \cF_s\right] \\
     &= \bE_{\nu}\left[ \left. {\sf A}^{f}(\infty) - \int_{t}^{\infty} e^{-u} f(X_u) du
	       + e^{-t} R_1 f(X_t) \, \right| \cF_s\right] \\
     &= \bE_{\nu}\left[{\sf A}^f(\infty) \mid \cF_s\right] -
      \bE_{\nu}\left[ \left. \int_{0}^{\infty} e^{-(u+t)} f(X_{u+t}) du
		\, \right| \cF_s\right] + 
		\bE_{\nu}\left[ e^{-t} R_1 f(X_t) \mid \cF_s \right] \\
     &=: \bE_{\nu}\left[{\sf A}^f(\infty) \mid \cF_s\right] - {\rm (I)} + {\rm (II)}.
   	\end{align*}
     By the Markov property,
     \begin{align*}
      {\rm (II)} &= 
      \bE_{\nu}\left[
      \left. e^{-t} \int_{0}^{\infty} e^{-u}
				\bE_{X_t}\left[f(X_u)\right] du
				\, \right| \cF_s\right]
      \\
      &= \bE_{\nu}\left[ \left. e^{-t} \int_{0}^{\infty}
		    e^{-u}\bE_{x}[f(X_{u+t})|\cF_t] du
		    \, \right| \cF_s\right] \\
      &= \bE_{\nu}\left[ \left. \int_{0}^{\infty} e^{-(u+t)} f(X_{u+t}) du
		    \, \right| \cF_s\right] = {\rm (I)}.
     \end{align*}
     Thus we know that $\bE_{\nu}[{\sf M}^f(t) \mid \cF_s] =
     \bE_{\nu}[{\sf A}^f(\infty) \mid \cF_s]$.
     Then, 
     \begin{align*}
	 \bE_{\nu}[{\sf A}^f(\infty) \mid \cF_s]
      &= \bE_{\nu} \left[ \left. \int_{0}^{s}
						 e^{-u} f(X_u) du +
						 \int_{s}^{\infty}
						 e^{-u}f(X_u) du 
						 \, \right|
						 \cF_s\right] \\
      &= {\sf A}^f(s) + \bE_{\nu} \left[ \left. 
      	\int_{s}^{\infty} e^{-u} f(X_u) du
			      \, \right| \cF_s \right] ,
	\end{align*}
	where by using the Markov property, 
	the above second term is rewritten by 
	\[
      \bE_{\nu}\left[ \left. \int_{0}^{\infty} e^{-(u+s)}
			     f(X_{u+s}) du \, \right| \cF_s \right] 
      = e^{-s}\int_{E}\bE_{X_s}
       \left[\int_{0}^{\infty}
	e^{-u}f(X_u)du \right] \nu(dx) 
      = e^{-s} R_1 f(X_s) .
     \]
     That is, $\bE_{\nu}[{\sf A}^f(\infty) \mid \cF_s] = {\sf M}^f (s)$. 
     Therefore we can find that
     $\{{\sf M}^f(t)\}_{t \ge 0}$ is an $(\cF_t, \bP_{\nu})$-martingale. 
    \end{proof}

    \begin{lemma}
     \la{lem-pre-03}
Let 
 $\{f_k\}_{ k \in \bN} \subset L^2(E;m) \cap {\cB}_+(E)$
and   ${\sf M}^{f_k}(t)$ be 
given by (\ref{newmar}) with $f_k$ in place of $f$.
     If $\{R_1 f_k\}_{k \in \bN}$ converges to a function 
     in $\cE_1$-strongly, 
     then  there exists a subsequence $\{ k_j \}$ of $\{ k\}$ such that
     \[
     \bP_x\left( {\sf M}^{f_{k_j}}(t)\ \text{converges uniformly in $t$ on $[0,\infty)$ as
     $j \to \infty$}\right) = 1, \quad \qe\ x \in E.
     \]
    \end{lemma}
    \begin{proof}
     Let $\nu$ be in $\cS_{00}$ with $\nu(E) = 1$. 
     Remark that ${\sf M}^{f_k}(\infty) = {\sf A}^{f_k}(\infty)$ in $L^2( {\bP}_\nu)$ and 
     $\{{\sf M}^{f_k}(t)\}_{t \ge 0}$ is an $(\cF_t,\bP_{\nu})$-martingale.
     By Doob's inequality and Lemma \ref{lem-pre-01},
     \[
      \bP_{\nu}\left(\sup_{0 \le t \le \infty}
     |{\sf M}^{f_k}(t) - {\sf M}^{f_{\ell}}(t)| > \varepsilon\right)
     \le \frac{1}{\varepsilon^2}
     \bE_{\nu} \left[\left({\sf A}^{f_k}(\infty) - {\sf A}^{f_{\ell}}(\infty)\right)^2\right] 
      \le  \frac{ M}{\varepsilon^2}
        \|R_1 f_k - R_1 f_{\ell}\|_{\cE_1},
     \] 
 where $M=M_\nu \sup_{k \in \bN} \| R_1 f_k \|_{\cE_1} \ (< \infty)$.
     Since $\{R_1 f_k\}_{k \ge 1}$ is an $\cE_1$-Cauchy sequence,
     there exists a subsequence $\{k_j\}$ of $\{k\}$
     such that
     \[
      \|R_1 f_{k_{j+1}} - R_1 f_{k_j}\|_{\cE_1} < 2^{-3j}.
     \]
     Now we set
     \[
     \Lambda_j = \left\{\omega \in \Omega :
     \sup_{0 \le t \le \infty} \left|{\sf M}^{f_{k_{j+1}}}(t,\omega) -
     {\sf M}^{f_{k_j}}(t,\omega)\right| > 2^{-j}\right\},
     \]
 and  we know that  $\bP_{\nu}(\Lambda_j) \le M \cdot 2^{-j}$, hence, 
 $\sum_{j=1}^{\infty} \bP_{\nu}(\Lambda_j) \le M < \infty$ holds. 
     Then the first Borel-Cantelli lemma provides us that 
     \[
     \bP_{\nu}\left(\limsup_{j \to \infty} \Lambda_j\right) = 0.
        \]
     For any $\mu \in \cS_{00}$, by setting $\nu(\cdot) = \mu(\cdot)/\mu(E)$,
     \begin{equation}
      \bP_{\mu}\left(\limsup_{j \to \infty} \Lambda_j\right) = 0.
       \la{eq-BC-1}
     \end{equation}
     Combining \eqref{eq-BC-1} with Theorem \ref{thm-equi}, we get
     this lemma. 
    \end{proof}

\medskip
    
    Now, we need to take subsequences repeatedly, so to avoid
    confusion with notations, we introduce a space of functions
    which define subsequences. 
    We denote by $\Phi$ all functions from $\bN$ to itself which are
    strictly increasing. Remark that for $\phi \in \Phi$, $\{\phi(n)\}$ is
    nothing but a subsequence of the sequence $\{n\}$. 

    \begin{lemma}
     \la{lem-pre-new}
Let  ${\sf A}^n, \, {\sf A} \in \bfA_c^+$.  Denote by $\mu_n$ and $\mu $ their Revuz measures 
in ${\cal S}$ for $n\ge1$. Assume that $\mu_n=f_n m$ and $\mu=fm$  for some $f_n, \, f  \in L^2(E;m) \cap {\cB}_+(E)$.
If $\lim_{n \to \infty}\rho(\mu_n, \mu) = 0$, then there exists a $ \phi \in \Phi$  such that 
     \[
     \bP_x\left( \lim_{n \to \infty} {\sf A}^{\phi(n)}_t  ={\sf A}_t  \text{ locally uniformly in $t$ on}\ [0,\infty) \right) = 1, \quad \qe x \in E.
     \]
 \end{lemma}
\begin{proof} 
We use an idea in the proof of \cite[Theorem~5.1.1]{FOT}.
For each $n \in {\bN}$, we put
$$
{\sf B}_n(t):=\int_0^t e^{-s} f_n(X_s) ds \quad {\rm and} \quad 
{\sf M}_n(t):= {\sf B}_n(t) + e^{-t} R_1 f_n(X_t).
$$
Since 
$\lim_{n \to \infty }\| R_1 f_n -U_1 \mu \|_{\cE_1}=0$,
by Lemma~\ref{lem-pre-03}, 
there is a $\phi_1 \in \Phi$ such that
\[
	{\bP}_x \left( 
		{\sf M}_{\phi_1(n)}(t) \text{ converges uniformly in $t$ on } 
		[0,\infty) \text{ as } n \to \infty 
	\right) 
	=1, 
	\quad \mbox{q.e.}\ x \in E.
\]
Further, by means of \cite[Lemma~5.1.2(i)]{FOT}, there is a $\phi_2 \in \Phi$ such that
\[
	{\bP}_x ( R_1 f_{\psi_2(n)}(X_t)
	\text{ converges locally uniformly in $t$ on}  \ [0,\infty) \text{ as } n \to \infty )
	=1, \quad \mbox{q.e.}\ x \in E,
\]
where $\psi_2=\phi_2 \circ \phi_1$.
Since ${\sf M}_n(\infty)={\sf B}_n(\infty)$ in $L^2( {\bP}_\nu )\ ( n\in {\bN},\ \nu \in {\cS}_{00})$,
there exists an exceptional set
     ${N} \subset E$ such that $\bP_x(\Lambda) = 1$
     for any $x \in E \setminus {N}$, where
\begin{align*}
	\tilde{\Omega} 
	&= 
	\big\{ 
		\omega \in \Omega 
		\mid 
		t \ge 0,\ X_t(\omega) \in E \setminus {N},
		X_{t-}(\omega) \in E \setminus {N} 
	\big\}, 
	\\
	\Lambda 
	&= 
	\left\{ 
		\omega \in \tilde{\Omega} \, 
		\left| \, 
		\begin{aligned}
			& {\sf B}_{\psi_2(n)}(\infty, \omega) < \infty \text{ and} \\ 
			& {\sf B}_{\psi_2(n)} (t,\omega) \text{ converges locally uniformly in $t$ on} \ [0,\infty)\ \text{as}\  n \to \infty 
		\end{aligned}
 		\right. 
 	\right\}.  
\end{align*}

     Now we define $ {\sf B}(t,\omega)$ and $\tilde{\sf A}(t,\omega)$ by 
     \begin{align*}
    {\sf B}(t,\omega) 
    &=
     \begin{cases}
      \ds \lim_{n \to \infty} {\sf B}_{\psi_2(n)}(t,\omega), & \omega \in \Lambda, \\
      \ds 0, & \omega \not\in \Lambda,
     \end{cases} \\
	\tilde{\sf A}(t,\omega) 
	&= \int_0^t e^s {\sf B}(ds,\omega),\quad \omega \in \tilde{\Omega}.
    \end{align*}
Noting 
\[
	{\sf A}^n_t=\int_0^t f_n(X_s)ds =\int_0^t e^s {\sf B}_n(ds)=e^t {\sf B}_n(t)-\int_0^t e^s {\sf B}_n(s)ds
\] and
$\tilde{\sf A}(t)=e^t{\sf B}(t) - \int_0^t e^s {\sf B}(s) ds$,
we find
  \begin{equation}
     \bP_x\left(\text{${\sf A}^{\psi_2(n)}_t$ 
     converges to $\tilde{\sf A}_t$ locally uniformly in $t$ on $[0,\infty)$ as $n \to \infty$}\right) = 1, \quad \text{q.e. } x \in E.
     \end{equation}
We claim the following:
\begin{equation}\label{newAA}
{\bE}_\nu \left[ \int_0^\infty e^{-t} d \widetilde{\sf A}(s)\right]= \int_E \widetilde{U_1 \mu}(x) \nu(dx), 
\quad \nu \in {\cS}_{00}.
\end{equation}
Note that
\begin{equation}\label{newAA1}
\lim_{t \to \infty} e^{-t}  \limsup_{n \to \infty} {\bE}_\nu \left[ R_1 f_n (X_t) \right]=0,
\end{equation}
which follows from
\begin{align*}
{\bE}_\nu \left[ R_1 f_n (X_t) \right] &= \int_E p_t(R_1 f_n)(x) \nu(dx) 
={\cE}_1 (p_t(R_1 f_n), U_1 \nu )\leq \| p_t(R_1 f_n )\|_{\cE_1} \| U_1 \nu \|_{\cE_1}\\
& \leq \| R_1 f_n \|_{\cE_1} \|U_1 \nu \|_{\cE_1}
\leq \Big(\sup_n \| R_1 f_n \|_{\cE_1} \Big) \|U_1 \nu \|_{\cE_1} <\infty.
\end{align*}
We next note that
\begin{align*}
\lim_{n \to \infty} {\bE}_\nu [ {\sf M}_{\psi_2(n)}(t) ] &=
 \lim_{n \to \infty} {\bE}_\nu [ {\sf M}_{\psi_2(n)}(\infty) ]=
 \lim_{n \to \infty} {\bE}_\nu [ {\sf B}_{\psi_2(n)}(\infty) ]=
 \lim_{n \to \infty} {\bE}_\nu \left[ \int_0^\infty e^{-s} f_{\psi_2(n)}(X_s)ds \right]\\
& =\lim_{n \to \infty}\int_E R_1 f_{\psi_2(n)}(x) \nu(dx)
=\lim_{n \to \infty}  {\cE}_1 (  R_1 f_{\psi_2(n)}, U_1 \nu )
=\int_E \widetilde{ U_1 \mu} (x) \nu(dx),
\end{align*}
and hence
\begin{align*}
{\bE}_\nu[B(t)] &=
 \lim_{n \to \infty} {\bE}_\nu [ {\sf B}_{\psi_2(n)}(t) ]=
 \lim_{n \to \infty} \Big( {\bE}_\nu [ {\sf M}_{\psi_2(n)}(t) ] -
  e^{-t}   {\bE}_\nu \left[ R_1 f_{\psi_2(n)} (X_t) \right]\Big) \\
&
=\int_E \widetilde{ U_1 \mu} (x) \nu(dx)
-e^{-t} \lim_{n \to \infty}   {\bE}_\nu \left[ R_1 f_{\psi_2(n)} (X_t) \right].
\end{align*}
Combining this with (\ref{newAA1}), we have
\[
{\bE}_\nu [{\sf B}(\infty)]=
\int_E \widetilde{ U_1 \mu} (x) \nu(dx),
\]
which implies (\ref{newAA}).
Since (\ref{newAA}) implies ${\sf A} \sim \widetilde{\sf A}$, we obtain the lemma.
\end{proof}

    \begin{proof}[Proof of Theorem \ref{thm-rho}]
     Since the $1$-potential $U_1 \mu_n$ of $\mu_n$ 
     is $1$-excessive,  we can find a Borel measurable and quasi-continuous modification
     $u_n$ of $U_1 \mu_n$
     and properly exceptional set  $N_n$ such that
     \[
       \begin{cases}
       kR_{k+1} u_n(x) \uparrow u_n(x)\quad (k \to \infty), & x \in E \setminus  N_n \\
       u_n(x) = 0, & x \in  N_n.
       \end{cases}
     \]
     Now we put
     \[
     g_n^k(x) :=
     \begin{cases}
      k(u_n(x) - kR_{k+1}u_n(x)), & x \in E \setminus  N_n \\
      0, & x \in  N_n.
     \end{cases}
     \]
     By Lemma \ref{lem-strong}, we know 
     that $R_1 g_n^k \to u_n,$ $\cE_1$-strongly as $k \to \infty$.
     Now let
     \[
      {\sf A}_n^k(t) := \int_{0}^{t}  g_n^k(X_s) ds.
\]
By means of Lemma~\ref{lem-pre-new}, for each $n \in \bN$, 
there exist $\psi^n \in \Phi$ such that
\[
{\bP}_x \left(
{\sf A}^{\psi^n(k)}_n(t) \ \text{ converges to}\ {\sf A}^n_t\ \text{locally uniformly in}\ t\ 
\text{on}\ [0,\infty)\ \text{as}\ k \to \infty
\right)=1, \quad \text{q.e. }  x \in E.
\] 
Let $T>0$, $\varepsilon >0 $ and $\nu \in {\cS}_{00}$.
Then there are numbers $k^n_o \in {\bN}$ such that
     \[
     \bP_{\nu}\left( 
     	\sup_{0 \leq t \leq T}
     	\left| {\sf A}_n^{\psi^n(k)}(t) - {\sf A}^n_t \right| 
     	> \frac{\varepsilon}2 
     \right) < 1/n,\quad
	\text{for any }  k \geq k^n_o,\ n \in {\bN}.
\]
Now, since $R_1 g_n^{ \psi^n(k)}$ converges to $u_n$ in $\cE_1$-strongly as
     $k \to \infty$ for each $n \in {\bN}$, 
there is a $k^n_* \geq k^n_o$ such that
     \begin{equation}
      \la{thm-08}
       \left\| R_1 g_n^{\psi^n(k)} - u_n\right\|_{\cE_1} < 1/n
     \quad \text{for any}\ k \ge k^n_*.
     \end{equation}
     Let $u$ be a quasi continuous modification of $U_1 \mu$.
     Putting  $h_n := g_n^{ \psi^n(k^n_*)}$,   we get
     \begin{equation*}
      \|R_1 h_n - u\|_{\cE_1} \le \|R_1 h_n - u_n\|_{\cE_1}
       + \|u_n - u\|_{\cE_1} < \frac{1}{n} + \|u_n - u\|_{\cE_1}
       \to 0 \quad (n \to \infty).
     \end{equation*}
     Hence we know that $R_1 h_n$ converges to $u$ in $\cE_1$-strongly
     as $n \to \infty$, that is, $\rho(h_n m, \mu) \to 0$. 
By means of Lemma~\ref{lem-pre-new},
there is a $\phi_* \in \Phi$ such that
     \begin{equation*}
      \begin{split}
       \bP_x\left(\text{${\sf A}_{n_*}^{\psi^{n_*}(k_*^{n_*})}(t)$ converges to ${\sf A}_t$
	  locally uniformly in $t$ on $[0,\infty)$ as $ n \to \infty$}\right)
       = 1,  \quad \qe x \in E,
      \end{split}
     \end{equation*}
where $n_*=\phi_*(n)$.
Therefore we obtain
     \begin{align*}
    & \!\!\!   \bP_{\nu} \left(\sup_{0 \le t \le T} \left| {\sf A}_t^{n_*} - {\sf A}_t \right| > \varepsilon \right) \\
      &\le \bP_{\nu}\left(\sup_{0 \le t \le T}
		    \left| {\sf A}_t^{n_*} - {\sf A}_{n_*}^{\psi^{n_*}(k_*^{n_*})}(t) \right| > \frac{\varepsilon}2\right)
      + \bP_{\nu}\left(\sup_{0 \le t \le T} \left| {\sf A}_{n_*}^{\psi^{n_*}(k_*^{n_*})}(t)- {\sf A}_t \right|
	       > \frac{\varepsilon}2\right) \to 0\ (n \to \infty).   
     \end{align*}
     Finally applying the first Borel-Cantelli lemma and
     Theorem \ref{thm-equi}, we arrive at the theorem.
    \end{proof}

       \section{ Compactness of Revuz maps in $\cS$}\label{sec5}
 
       In this section, we discuss the convergence of {\sf PCAF}s associated with  general smooth measures.

    \begin{theorem}
     \la{thm-main-02}
Let  ${\sf A}^n, \, {\sf A} \in \bfA_c^+$.  Denote by $\mu_n$ and $\mu $ 
their Revuz measures in ${\cal S}$ for $n\ge1$.  
     Assume there exists an increasing sequence of closed sets $\{F_k\}$  such that
     \begin{equation}
      \Cap\left(E \setminus \bigcup_{k = 1}^{\infty}F_k \right) = 0;
       \la{eq-condi-0}
     \end{equation}
     \begin{enumerate}[{\rm (Sa)}]
      \item for each $n, k  \in \bN$, $1_{F_k} \mu_n,
	    1_{F_k} \mu \in \cS_0$ and, for each $k\in \bN$, 
	    \begin{equation}
	     \lim_{n \to \infty} \rho(1_{F_k} \mu_n, 1_{F_k}  \mu) = 0;
	      \la{eq-condi-a}
	    \end{equation}
      \item for any $\nu \in \cS_{00}$,
	    \begin{equation}
	     \int_{F_k^c} \widetilde{U_1 \nu} \, d\mu < \infty\ \text{holds for any $k\in \bN$} \quad {\rm and}  \quad
	      \lim_{k\to \infty} \int_{F_k^c} \widetilde{U_1 \nu} \, d\mu = 0;
	      \la{eq-condi-b}
	    \end{equation}
      \item for any $\nu \in \cS_{00}$,
	    \begin{equation}
	    \sup_{n\in{\mathbb N}}  \int_{F_k^c} \widetilde{U_1 \nu} \, d\mu_n < \infty\ \text{holds for any 
	    $k\in \bN$} \quad {\rm and} \quad
	      \lim_{k\to \infty} \sup_{n \in \bN} 
	      \int_{F_k^c} \widetilde{U_1 \nu} \, d\mu_n = 0.
	      \la{eq-condi-c}
	    \end{equation}
     \end{enumerate}
     Then there exists a subsequence $\{n_k\}$ of $\{n\}$ such that
     \begin{equation}
      \bP_x\left(  \lim_{n_k \to \infty} {\sf A}_t^{n_k} ={\sf A}_t\ \text{ locally uniformly in $t$ 
      on $[0,\infty)$}\right) = 1, \quad \qe x \in E.
     \end{equation}
    \end{theorem}
    \begin{proof}
     Put $\Gamma_{k} = F_{k + 1} \setminus F_{k}$,
     $\mu_n^{(k)} = 1_{\Gamma_k} \mu_n$ and  $\mu^{(k)} = 1_{\Gamma_k} \mu$ $(k, n \in {\bN})$.	
By means of (Sa), 
     $\mu_n^{(k)}, \mu^{(k)} \in {\cS}_0$.
Let ${\sf A}^{(k)}_n,\ {\sf A}^{(k)}$ be the {\sf PCAF}s corresponding to   
$\mu_n^{(k)}, \mu^{(k)}$, respectively.
By virtue of \cite[Theorem~5.1.3]{FOT},
${\sf A}^{(k)}_n \sim 1_{\Gamma_k} {\sf A}_n$ and
${\sf A}^{(k)} \sim 1_{\Gamma_k} {\sf A}$.

Let  $\varepsilon > 0,\ T > 0$ and
     $\nu \in \cS_{00}$.
We show that
 there exists a subsequence $\{n_k\}$ of $\{n\}$ such that
     \begin{equation}
      \lim_{n_k \to \infty} \bP_{\nu} \left(\sup_{0 \le t \le T} \left|{\sf A}_t^{n_k} - {\sf A}_t\right| >
				     \varepsilon\right) = 0.
      \label{eq:main-02-2}
     \end{equation}
To prove this,  divide the above into three terms for $j, n$ as follows: 
     \begin{align*} 
\bP_{\nu} \left(\sup_{0 \le t \le T} \left|{\sf A}_t^n - {\sf A}_t\right| > \varepsilon\right)
      &\le   \ \bP_{\nu} \left(\sup_{0 \le t \le T} \left| {\sf A}_t^n - 1_{F_j} {\sf A}_t^n \right|
		       > \frac{\varepsilon}{3}\right) + \bP_{\nu} \left(\sup_{0 \le t \le T} \left|1_{F_j} {\sf A}_t^n - 1_{F_j} {\sf A}_t \right| >  \frac{\varepsilon}{3}\right)  \\
		       &  \ \quad +   \bP_{\nu}\left(\sup_{0 \le t \le T}
		\left| 1_{F_j} {\sf A}_t - {\sf A}_t \right| > \frac{\varepsilon}{3}\right) \\
      &=: {\rm (I)}_{j,n} + {\rm (II)}_{j,n} + {\rm (III)}_j. 
     \end{align*}
     By the Chebyshev inequality and the inequality, $1\le e^{T-s}$ for $0\le s\le T$,  the term ${\rm (III)}_j$ is estimated as 
          \[
      {\rm (III)}_j \le \frac{3e^T}{\varepsilon} \bE_{\nu}
            \left[ \int_{0}^{T} e^{-s} 1_{F_j^c} (X_s) d{\sf A}_s\right]
           = \frac{3e^T}{\varepsilon} \sum_{k=j}^{\infty} \bE_{\nu}
            \left[\int_{0}^{\infty} e^{-s} 1_{\Gamma_{k}}(X_s) d{\sf A}_s \right].
     \]
      Then the expectation part is rewritten as follows: 
     \begin{align*}
     \bE_{\nu}
                 \left[\int_{0}^{\infty} e^{-s} 1_{\Gamma_{k}}(X_s) d{\sf A}_s \right] 
      &= 
       \int_{E} \bE_{x}\left[\int_{0}^{\infty} e^{-s} d{\sf A}_s^{(k)}\right] \nu(dx)
       = 
       \int_{E} \widetilde{U_1 \mu^{(k)}} d\nu \\
      &= 
       \cE_1 \left( U_1 \mu^{(k)}, U_1 \nu \right) 
       = 
       \int_{E} \widetilde{U_1 \nu} \, d\mu^{(k)}
       =
       \int_{\Gamma_k} \widetilde{U_1 \nu} \, d \mu .
       \end{align*}
       Thus, we have that 
       \[
       	 {\rm (III)}_j \le 
       	 \frac{3e^T}{\varepsilon} \sum_{k=j}^{\infty}
       	 \int_{\Gamma_k} \widetilde{U_1 \nu} \, d \mu
       	 =
       	 \frac{3e^T}{\varepsilon} \int_{F_j^c} \widetilde{U_1 \nu} \, d\mu .
       \]

Let us fix an $\varepsilon'>0$ arbitrarily.
Then, by means of (Sb), there is an $j_1$ such that ${\rm (III)}_j < \varepsilon'$ for any $j \geq j_1$.
In the same way as above, we get the following:
\[
      {\rm (I)}_{j,n} \le
      \frac{3e^T}{\varepsilon} \sup_{n \in {\bN}}\int_{F_j^c} \widetilde{U_1 \nu} \, d\mu_n.
\]
By means of (Sc), 
there is an $j_2\ ( \geq j_1)$ such that ${\rm (I)}_{j,n} < \varepsilon'$ for any $j \geq j_2, n \in {\bN}$.
     Finally we consider ${\rm (II)}_{j,n}$ with $j=j_2$.
Noting (Sa) and using Theorem~\ref{thm-rho}, we find a subsequence $\{n_k\}$ such that 
$\lim_{n_k \to \infty}  {\rm (II)}_{j_2,n_k} =0$.
Thus we obtain
     \begin{equation*}
      \limsup_{n_k \to \infty} \bP_{\nu} \left(\sup_{0 \le t \le T} |{\sf A}_t^{n_k} - {\sf A}_t| >
				     \varepsilon\right) \leq 2 \varepsilon'.
           \end{equation*}
Letting $\varepsilon' \downarrow 0$ leads us to (\ref{eq:main-02-2}).
Applying the first Borel-Cantelli lemma and
     Theorem \ref{thm-equi}, we arrive at the theorem.
    \end{proof}

    \begin{remark}
\begin{itemize}
\item[(i)]      \la{rem-02}

     If $\mu_n$ and $\mu$ are absolutely continuous with respect to
     $m$, then the assumption \eqref{eq-condi-0} can be replaced by
     \[
      m\left(E \setminus \bigcup_{k=1}^{\infty}F_k\right) = 0. 
     \]
%

\item[(ii)]  \la{rem-03}
 If $\mu_n, \mu \in \cS_0$ satisfy the assumption of Theorem~\ref{thm-rho}, 
that is, $ \lim_{n \to \infty}\rho(\mu_n, \mu) = 0$,
then the assumption of Theorem~\ref{thm-main-02} is satisfied with $F_k=E\ (k\in {\bN})$. 

\end{itemize}
\end{remark}

    \section{A class of absolutely continuous measures that are smooth}\label{sec6}

In this section we define a class of smooth measures that are absolutely continuous with respect to the
basic measure $m$.  We say that a measurable function $f$ on $E$ is {\it locally in $L^1(E;m)$ in the broad sense} 
($u \in \dot{L}^1_{\sf loc}(E;m)$ in notation) if there exists an increasing sequence of compact sets $\{F_n\}$ such that
$\bigcup_{n=1}^\infty F_n=E$ {\it q.e.} and $f  \in L^1(F_n;m)$ for any $n$.  It is clear that 
$L^p(E;m) \subset L^1_{\loc}(E;m) \subset \dot{L}_{\loc}^1(E;m)$ for any $p\in [1, \infty]$.

   \begin{lemma}
     \la{lem-S0}
     Assume that $f \in \dot{L}_{\sf loc}^1(E;m) \cap \cB_+(E)$ and
     put $\mu := f m$. Then $\mu$ is a Radon measure and belongs to $\cS$.
    \end{lemma}
       \begin{proof}
Let $ f \in \dot{L}_{\sf loc}^1(E;m) \cap \cB_+(E)$ and $\{ F_n \}$ be an increasing sequence of compact sets satisfying above condition.
	By the assumption of $\{F_n\}$,
	it is obvious that $\mu(F_n) < \infty$ for any $n \in \bN$.
	For any compact set $K$,
	there exists a number
	$N \in \bN$ such that $K \subset F_N, \qe$
	Note that 
	it holds that $m(A) = 0$ for any $A \subset E$ with $\capa(A) = 0$.
	Since $K \setminus F_N$ is $m$-negligible, 
	\[
	\mu(K) = \int_{K} fdm = \int_{K \cap F_N} fdm + \int_{K \setminus F_N}
	fdm = \int_{K \cap F_N} fdm < \infty.
	\]
	Thus we see that $\mu$ is a Radon measure, and moreover,
	we also obtain 
	\[
	\lim_{n \to \infty} \capa(K \setminus F_n) \le \capa(K \setminus F_N) = 0.
	\]
	Hence $\mu = f  m$ is a smooth measure. 
       \end{proof}

    \begin{definition}
     \la{def-A}
     Write $\frA$ for the set of functions $f$ satisfying that $f \in
     \dot{L}_{\sf loc}^{1}(E;m) \cap \cB _+ (E)$ and 
     there exist an sequence of functions $\{f_n\} \subset
     \dot{L}_{\sf loc}^1(E;m) \cap \cB_+(E)$, an increasing sequence of 
     closed sets $\{F_n\}$ satisfying $ m \left( E \setminus \bigcup_{n=1}^\infty F_n \right)=0$, 
     and sequences $\{p_n\}, \{q_n\}, \{r_n\}$ with
     $1 \le p_n, q_n, r_n < \infty$ such that 
\begin{itemize}
	     \item[(Aa)]  $\ds \lim_{n \to \infty} \|f_n - f\|_{L^{p_n}(F_n)} = 0;$
      \item[(Ab)] one of the following conditions is satisfied;
	    \begin{itemize}
	     \item[(Ab1)] \ $\ds \lim_{n \to \infty} \sum_{k=n}^{\infty}
			 \|f_n\|_{L^{q_k}(F_{k+1} \setminus F_k)}
			 \cdot (1 \wedge m(F_{k+1} \setminus F_k)^{1/q_k'}) = 0$,
	     \item[(Ab2)] \ $\ds \lim_{n \to \infty} C^{q_n}
			 \|f_n\|_{L^{q_n}(E \setminus F_n)}^{q_n} = 0$ for
			 any $C > 1$, 
	    \end{itemize}
      \item[(Ac)] one of the following conditions is satisfied;
	    \begin{itemize}
	     \item[(Ac1)] \ $\ds \lim_{n \to \infty} \sum_{k=n}^{\infty}
			 \|f\|_{L^{r_k}(F_{k+1} \setminus F_k)}
			 \cdot (1 \wedge m(F_{k+1} \setminus F_k)^{1/r_k'}) = 0$,
	     \item[(Ac2)] \ $\ds \lim_{n \to \infty} C^{r_n}
			 \|f\|_{L^{r_n}(E \setminus F_n)}^{r_n} = 0$ for
			 any $C > 1$, 
	    \end{itemize}
\end{itemize}
     where $q_k'$ (resp. $r_k'$) satisfies that $1/q_k + 1/q_k' = 1$ 
     (resp. $1/r_k + 1/r_k' = 1$). 
     We call the sequence $\{f_n\}$ an
     {\it $\frA$-approximating sequence of $f$}.
    \end{definition}

    \begin{remark}
     \la{rem-01}
\begin{itemize}
\item[(i)] There exists  an example such that the condition (Ab1)  (resp. (Ab2)) is satisfied   but  (Ab2) 
 (resp. (Ab1))  fails  for some $\{q_n\}$ and $\{F_n\}$.  
 Similarly, there exists an example so that the same conclusions go for (Ac1) and (Ac2) to some $\{r_n\}$ and $\{F_n\}$  (see \S 7.1).


\item[(ii)] Let $ f_n,\, f \in  L^p(E;m) \cap \cB _+(E) \ ( n\in {\bN}) $ such that 
$ \| f_n - f \|_{L^p} \to 0$ as $ n \to \infty$, where $ 1 \leq p <\infty$.
Then $\{f_n\}$ is an $\frA$-approximating sequence of $f$.
 Indeed,  taking an increasing sequence $\{ F_n \}$
of compact sets satisfying $m( E \setminus \bigcup_{n=1}^\infty F_n)=0$,
it is easy to see that 
$ f_n,\, f \in \dot{L}_{\sf loc}^1(E;m) \cap \cB_+(E)$ and the conditions  (Aa),  (Ab2) and (Ac2) are satisfied 
with $p_n = q_n = r_n = p \, (n \in {\bN})$.  So, $L^p(E;m) \cap \cB _+(E) \subset \frA$ holds for $ 1 \leq p < \infty$. 

\item[(iii)] We note that there is an $f \in  \frA$ such that $ f \not\in L^p(E;m)$ for any $p \in [1,\infty]$. 
 We  emphasize that the integrability exponents $p_n, q_n, r_n$ are allowed to vary for functions of $\frA$ (see also \S7.1).

\end{itemize}
\end{remark}


    \begin{theorem}
     \label{thm-02}
     Let $f \in \frA$ and $\{f_n\}$ be an $\frA$-approximating sequence
     of $f$. Let $\{{\sf A}_t^n\}$ and $\{{\sf A}_t\}$ be {\sf PCAF}s corresponding to 
     $f_n  m$ and $f m$, respectively.
     Then there exists a subsequence $\{f_{n_k}\}$ of $\{f_n\}$
     such that 
     \[
     \bP_{x}\left( \text{${\sf A}_t^{n_k}$ converges to ${\sf A}_t$ locally uniformly
     in $t$ on $[0,\infty)$}\right) = 1, \quad 
     {\rm q.e.} \ x \in E.
     \]
    \end{theorem}


    \begin{proof}
     Take $\varepsilon > 0$ and $T > 0$. For $\nu \in \cS_{00}$, we have
     \begin{align*}
      \bP_{\nu}\left(\sup_{0 \le t \le T} \left| {\sf A}_t^n - {\sf A}_t \right| >
	      \varepsilon \right) 
      &\le 
      \bP_{\nu} \left(\sup_{0 \le t \le T} \left| {\sf A}_t^n - 1_{F_n} \cdot {\sf A}_t^n \right| >
		     \frac{\varepsilon}{3}\right) 
      + \bP_{\nu}\left(\sup_{0 \le t \le T} \left| 1_{F_n} \cdot  {\sf A}_t^n - 1_{F_n} \cdot {\sf A}_t \right|
	       > \frac{\varepsilon}{3}\right) 
	   \\ &
	   \quad + \bP_{\nu}\left(\sup_{0 \le t \le T}
	       \left| 1_{F_n} \cdot {\sf A}_t - {\sf A}_t \right| > \frac{\varepsilon}{3}\right) \\
      &=: {\rm (I)} + {\rm (II)} + {\rm (III)}.
     \end{align*}
	Here $1_{F_n} \cdot {\sf A}_t ^n = \int _0 ^t  1_{F_n} (X_s) f_n (X_s) d s$.  
     We start with ${\rm (II)}$. By Chebyshev's and H\"older's inequalities, 
     \begin{align*}
      {\rm (II)} &= 
      \bP_{\nu} \left(\sup_{0 \le t \le T} \left|\int_{0}^{t} 1_{F_n} \cdot (f_n - f)(X_s) \, ds
							   \right| > \frac{\varepsilon}{3}\right) 
      \\ 
      & \le 
      \frac{3}{\varepsilon} \bE_{\nu}
       \left[\int_{0}^{T} 1_{F_n} (X_s) |f_n - f|(X_s) \, ds \right] \\
      & \le 
      \frac{3e^T}{\varepsilon} \bE_{\nu} \left[\left(\int_{0}^{T} e^{-s}
						 1_{F_n} (X_s) | f_n - f |(X_s) \, 
						 ds\right)^{p_n}\right]^{1/p_n} 
						 \\
	& = 
 	\frac{3e^T}{\varepsilon} \bE_{\nu} \left[\left(\int_{0}^{T} e^{(-s)(1/p_n +1/p'_n)}
						 1_{F_n} (X_s) | f_n - f |(X_s) \,
						 ds\right)^{p_n}\right]^{1/p_n} \\
	& \le
 \frac{3e^T}{\varepsilon} \bE_{\nu} \left[\left(\int_{0}^{T} e^{-s}
						 1_{F_n} (X_s) | f_n - f |^{p_n}(X_s) \,
						 ds\right)
\left(\int_0^T e^{-s} ds \right)^{p_n/p'_n}
\right]^{1/p_n} \\
&    \le  \frac{3e^T}{\varepsilon} \bE_{\nu} \left[\int_{0}^{\infty} e^{-s}
		1_{F_n} (X_s)|f_n - f|^{p_n}(X_s) \, ds\right]^{1/p_n}.
	\end{align*}

    By Lemma \ref{lem-pre-00},
     \begin{align}
      {\rm (II)} &\le  \frac{3e^{T}}{\varepsilon}
       \left(\int_{E} R_1 (1_{F_n}|f_n - f|^{p_n})(x) \, \nu(dx)
       \right)^{1/p_n} \nonumber \\
      &=  \frac{3e^{T}}{\varepsilon}
       \left(\int_{E} 1_{F_n}|f_n - f|^{p_n}(x) \cdot U_1 \nu(x) \, 
	m(dx)\right)^{1/p_n} \nonumber \\
      &\le \frac{3 e^{T} \|U_1
       \nu\|_{\infty}^{1/p_n}}{\varepsilon}
       \|f_n - f\|_{L^{p_n}(F_n)} \la{eq:thm-02-01}.
     \end{align}
Since $0 < 1/p_n \le 1$ by the 
     assumption of $\{p_n\}$, 
     we  find that
     \[
      \sup_{n \in \bN} \|U_1
     \nu\|_{\infty}^{1/p_n} < \infty. 
     \]
     Hence ${\rm (II)}$ 
tends to $0$ as
     $n \to \infty$ by the assumption (Aa) in Definition \ref{def-A}.

     Next we consider ${\rm (I)}$. First we assume (Ab1).
     Put $\Gamma_k = F_{k+1} \setminus F_k$. 
     Then, by the Chebyshev inequality, 
     \[
      {\rm (I)} \le \frac{3}{\varepsilon}
       \bE_{\nu}\left[\sup_{0 \le t \le T}
		 \int_{0}^{t}
		 1_{F_n^c} (X_s)
		 f_n(X_s) ds\right] \le 
       \frac{3}{\varepsilon} \sum_{k=n}^{\infty}
       \bE_{\nu}\left[\int_{0}^{T} 1_{\Gamma_k}(X_s)
		 f_n(X_s) ds\right].
     \]
	We obtain by the H\"older inequality and Lemma~\ref{lem-pre-00},
     \begin{align*}
 \bE_{\nu}\left[\int_{0}^{T} 1_{\Gamma_k}(X_s)
		 f_n(X_s) ds\right]
      &\le  \bE_{\nu}\left[\int_{0}^{T}
					   1_{\Gamma_k}(X_s)
					   f_n(X_s)^{q_k}
					   ds\right]^{1/q_k}
       \cdot \bE_{\nu}\left[\int_{0}^{T} 1^{q_k'}ds\right]^{1/q_k'} \\
      &\le  e^{T/q_k} (T \cdot \nu(E))^{1/q_k'}
       \cdot \bE_{\nu} \left[\int_{0}^{\infty} e^{-s} 1_{\Gamma_k} (X_s)
		  f_n(X_s)^{q_k} ds\right]^{1/q_k} \\
      &\le
 e^{T/q_k} (T \cdot \nu(E))^{1/q_k'}
       \left(\int_{E} 1_{\Gamma_k}(x) f_n(x)^{q_k} U_1
       \nu(x) \, m(dx)\right)^{1/q_k} \\ 
      &\le C_1   \|f_n\|_{L^{q_k}(\Gamma_k)},      
     \end{align*}
where $C_1 := \sup_{k \in {\bN}}  e^{T/q_k} (T \cdot \nu(E))^{1/q_k'} \| U_1 \nu \|_\infty^{1/q_k} < \infty$ by virtue of $ 1/q_k,\ 1/q'_k \in (0, 1]$.
     On the other hand, we can also estimate as follows:
     \begin{align*}
& \!\!\!\!  \bE_{\nu}\left[\int_{0}^{T} 1_{\Gamma_k}(X_s)
		 f_n(X_s) ds\right]
 \le       \bE_{\nu}\left[\int_{0}^{T} 1_{\Gamma_k}(X_s) f_n(X_s)^{q_k} 
		 ds\right]^{1/q_k} \cdot \bE_{\nu}\left[\int_{0}^{T}
		 1_{\Gamma_k}(X_s) \, ds \right]^{1/q_k'} \\
      &\le  e^{T(1/q_k +  1/q_k')} \left(\int_{E} R_1 (1_{\Gamma_k} f_n^{q_k})(x) \, \nu(dx)\right)^{1/q_k} \cdot
       \left(\int_{E} R_1 (1_{\Gamma_k})(x) \, \nu(dx)\right)^{1/q_k'} \\
      &= e^T \left(\int_{E} 1_{\Gamma_k} (x) f_n^{q_k}(x) U_{1} \nu (x) \, m(dx)\right)^{1/q_k} \cdot
       \left(\int_{E} 1_{\Gamma_k}(x) U_1 \nu(x) \, 
	m(dx)\right)^{1/q_k'} \\
      &\le e^T \|U_1\nu\|_{\infty}
        \|f_n\|_{L^{q_k}(\Gamma_k)}
       m(\Gamma_k)^{1/q_k'},
     \end{align*}
     Hence, by (Ab1) we can see the following.
\[
{\rm (I)} \leq \frac{3}{\varepsilon} \left( C_1 + e^T \| U_1 \nu \|_\infty \right) \sum_{k=n}^\infty \| f_n \|_{ L^{q_k} ( \Gamma_k ) }
\left( 1 \land m(\Gamma_k )^{1/q'_k}\right) \to 0\quad \mbox{as}\ n \to \infty.
\]

     We next assume (Ab2). By the Chebyshev inequality,
     \begin{align*}
      {\rm (I)} &\le \left(\frac{3}{\varepsilon}\right)^{q_n}
       \bE_{\nu}\left[\left(\int_{0}^{T} e^{T - s} 1_{F_n^c}(X_s)
		       f_n(X_s) \, ds \right)^{q_n}\right] \\
      &\le \left(\frac{3 e^{T}}{\varepsilon}\right)^{q_n}
       \bE_{\nu} \left[\left(\int_{0}^{\infty} e^{-s} 1_{F_n^c}(X_s)
			f_n(X_s) \, ds\right)^{q_n}\right] \\
      &\le \left(\frac{3e^{T}}{\varepsilon}\right)^{q_n}
       \bE_{\nu} \left[\int_{0}^{\infty} e^{-s} 1_{F_n^c}(X_s)
		  f_n(X_s)^{q_n} \, ds\right],
     \end{align*}
     where we use the Jensen inequality for the last inequality on
     account that $e^{-s} ds$ is the probability measure on
     $(0,\infty)$. Thus we can estimate (I) as follows:
     \begin{align*}
      {\rm (I)} &\le \left(\frac{3 e^{T}}{\varepsilon}\right)^{q_n}
       \int_{E} R_1(1_{F_n^c} f_n^{q_n})(x) \, \nu(dx)
      \le \left(\frac{3e^T}{\varepsilon}\right)^{q_n}
      \int_{E} 1_{F_n^c}(x) f_n^{q_n}(x) U_1 \nu(x)\, m(dx) \\
      &\le \|U_1 \nu\|_{\infty}
       \left(\frac{3e^T}{\varepsilon}\right)^{q_n}
       \|f_n\|^{q_n}_{L^{q_n}(F_n^c)}\ \to\ 0 \quad \text{as }n \to \infty.
     \end{align*}
     
     For ${\rm (III)}$, using the similar arguments to ${\rm (II)}$ with (Ac2),
     we can obtain that ${\rm (III)}$ tends to $0$ as $n \to \infty$.
     
     Thus, we know that
     \[
     \lim_{n \to \infty} \bP_{\nu}\left(\sup_{0 \le t \le T} |{\sf A}_t^n - {\sf A}_t| > \varepsilon\right) = 0. 
     \]

     Therefore, applying the first Borel-Cantelli lemma and
     Theorem \ref{thm-equi}, we obtain the theorem.
    \end{proof}


Note that, for a sequence of functions $\{f_n\}$ and $f$ satisfying the conditions of Theorem \ref{thm-02},  it is not clear 
in general whether  the measures $\{\mu_n\}:=\{f_nm\}$ and $\mu:=fm$ satisfy the conditions of Theorem \ref{thm-main-02} or not.
So, we consider the special case that 
the densities $f_n$ and $f$
    satisfy $f_n \to f$ in $L^p(E;m)$ for $ 1 \leq p < \infty$ in the rest of this section.
In particular,  the assumption of Theorem~\ref{thm-rho} is satisfied in this case (Remark \ref{rem-01} (ii)).
Here, we show that the assumption of Theorem~\ref{thm-main-02} is also satisfied.

    \begin{proposition}
     \la{prop-AC}
     Let $1 < p < \infty$ and $f_n, f \in L^p(E;m) \cap \cB_+(E)$ for $n
     \in \bN$. Put $\mu_n = f_n m$ and $\mu = f  m$. Assume that $f_n \to f$ in $L^p$.
     Then there exists an increasing sequence of compact subsets
     $\{F_k \}$ with
     \[
      m\left(E \setminus \bigcup_{k=1}^{\infty} F_k  \right) = 0
     \]
     such that conditions (Sa), (Sb), (Sc) in Theorem \ref{thm-main-02} are
     satisfied. 
    \end{proposition}
    \begin{proof}
The case with $p=2$  follows from Example~\ref{ex-density} and Remark~\ref{rem-03} (ii). 
The following proof is also available for that case.  Since $f_n, f \in L^p(E;m)
     \subset L_{\sf loc}^1(E;m)$, we know that $\mu_n, \mu$ are in $\cS$ by
     Lemma \ref{lem-S0}.
     The proof of this proposition is split into two cases on $1 < p \le
     2$ and $2 < p < \infty$. Before starting the proof,  take
     $\{K_k\}$ an increasing sequence of compact sets of $E$
     satisfying $E=\bigcup_{k=1}^\infty K_k$. 

     {\bf Case I.} $1 < p \le 2$: For any $\ell \in \bN$, set
     \[
      F_k = \ov{\left\{x \in K_k : \left(\sup_{n}
     f_n(x)\right) \vee f(x) \le k \right\}},
     \]
     where $\ov{A}$ denotes the closure of set $A$. 
     On account that $f_n \to f$ in $L^p(E;m)$, we can find that
     $\{F_k\}$ is an increasing sequence of compact sets of $E$
     so that $E = \bigcup_{k=1}^{\infty} F_k \ \mae$ Since, for $k \in \bN$ and $1 < p \le 2$,
     \begin{align} \nonumber
      \| 1_{F_k} f_n - 1_{F_k} f\|_2^2
       &=
       \int_{E} \left| 1_{F_k}(x) f_n(x) - 1_{F_k}(x)
		     f(x)\right|^2 m(dx) \\ \nonumber
      &=
      \int_{F_k} \left| f_n(x) - f(x)\right|^2 m(dx)
      \\  \nonumber 
      &\le  
       (2k)^{2-p} \int_{F_k} \left| f_n(x) - f(x)\right|^p
       m(dx) \\
      &\le (2k)^{2-p} \|f_n - f\|_p^p\ \to\ 0 \quad (n \to \infty), \label{pnew1}
     \end{align}
     it follows that 
     $1_{F_k} f_n$ and $1_{F_k} f$ belong to
     $L^2(E;m)$ and $1_{F_k} f_n \to 1_{F_k}f$ in $L^2(E;m)$
     as $n \to \infty$.  
      Noting that $U_1 (1_{F_k} \mu_n) = R_1(1_{F_k}f_n)$ and
     $U_1(1_{F_k}  \mu) = R_1(1_{F_k} f)$ for any $n$ and $k$, we see that $1_{F_k}  \mu_n$ and $1_{F_k} \mu$
     belong to $\cS_0$ for each $n$ and $k$, and hence (Sa) holds. Moreover, we find that
     \begin{align*}
      \rho(1_{F_k}  \mu_n, 1_{F_k}  \mu) 
      &=
       \cE_1\big(U_1(1_{F_k}  \mu_n) - U_1(1_{F_k} \mu),
       U_1(1_{F_k}  \mu_n) - U_1(1_{F_k}  \mu)\big) \\
      &= \cE_1\big(R_1(1_{F_k} f_n - 1_{F_k}f), R_1(1_{F_k} f_n - 1_{F_k}f)\big) \\
      &= \int_{E} \Big( 1_{F_k}(x) f_n(x) - 1_{F_k}(x) f(x)\Big) R_1\left( 1_{F_k} f_n - 1_{F_k}f \right)(x) \, m(dx) ,
      \end{align*}
     then by the Schwarz inequality, 
     \begin{align} \nonumber
     \rho(1_{F_k} \cdot \mu_n, 1_{F_k} \cdot \mu) 
      &\le 
      \|1_{F_k} f_n - 1_{F_k} f\|_2 \big\|R_1(1_{F_k}
       f_n - 1_{F_k} f)\big\|_2
       \\  \nonumber
       &\le
       \|1_{F_k} f_n - 1_{F_k}f\|_2^2 \\
      &\le (2k)^{2-p} \|f_n - f\|_p^p\ \to\ 0 \quad (n \to \infty),   \label{pnew2}
     \end{align}
     hence (Sa) holds. 

     To show (Sb) and (Sc), take any $\nu \in \cS_{00}$. Note that, 
     for any $2 \le q < \infty$, $U_1 \nu \in L^q(E;m)$ because
     $U_1 \nu \in L^2(E;m) \cap L^{\infty}(E;m)$. Then
     using the H\"older inequality with $1/p + 1/q = 1$ and $1 < p \le
     2 \le q < \infty$, we see that
     \[
      \int_{F_k^c} \widetilde{U_1 \nu}(x) \, \mu_n(dx)
       = \int_{F_k^c} \widetilde{U_1 \nu}(x) f_n(x) \, m(dx) 
      \le \|U_1 \nu\|_{q} \|f_n\|_p, 
     \]
     and the right hand side is uniformly bounded in $n$.
     Moreover,
     \begin{equation}
     \sup_n \int_{F_k^c} \widetilde{U_1 \nu}(x) \, \mu_n(dx)
     \le \left(\int_{F_k^c} (\widetilde{U_1 \nu}(x))^q
     \, m(dx)\right)^{1/q} \Big( \sup_n  \|f_n\|_p\Big) \to 0 \quad (k \to \infty)
     \label{eq:20240112}
     \end{equation}
     holds. This implies the condition (Sc).
     Replacing $f_n$ by $f$, we can prove (Sb).

     {\bf Case II.} $p > 2$:  We show (Sa), (Sb), (Sc) with $\{ F_k \}$ replaced by $\{ K_k \}$.
Noting that $1_{K_k} f$ and $1_{K_k} f_n$ belong to $L^2(E;m)$, we see that
$1_{K_k} \mu$ and $1_{K_k} \mu_n$ belong to ${\cS}_0$.  Take any $k$ and fix it. By the H\"older inequality, 
     \begin{align*}
      \rho(1_{K_k} \mu, 1_{K_k} \mu_n) 
      &=  \cE_1 \big(R_1(1_{K_k}(f_n - f)), R_1(1_{K_k}(f_n - f))\big) \\
      &\le \int_{E} 1_{K_k}(x) (f_n(x) - f(x))^2 \, m(dx)\\
      &\le \|f_n - f\|_p^2 \, m(K_k)^{1 - p/2}\ \to\ 0 \quad (n \to \infty).
     \end{align*}
Thus (Sa) is satisfied. We show (Sb).  Put
     \[
      H_a = \{x \in E : f(x) \le a\} \quad \text{for}\ a > 0. 
     \]
     For any $\varepsilon > 0$, there exists a constant $R \ge 1$ such
     that
$
      \int_{H_{R}^c} f(x)^p m(dx) < \varepsilon.
$  Let $\nu \in {\cal S}_{00}$. 
     Then for this constant $R$,
     \begin{align*}
       \int_{K_k^c} \widetilde{U_1 \nu}(x) \, \mu(dx) &=
	\int_{K_k^c} \widetilde{U_1 \nu}(x) f(x) \, m(dx) \\
      &= \left(\int_{K_k^c \cap H_R} + \int_{K_k^c \cap
	  H_R^c}\right) \widetilde{U_1 \nu}(x) f(x) \, m(dx) 
      =: {\rm (I)}_k + {\rm (II)}_k
     \end{align*}
     Because $R \ge 1$ and $p \ge 1$, we know that $f(x) \le f(x)^p$ on
     $H_R^c$, hence
     \begin{equation} \label{pnew3}
     {\rm (II)}_k  \le  \|U_1 \nu\|_{\infty} \int_{H_R^c} f^p dm
     < \varepsilon \|U_1 \nu\|_{\infty}.
  \end{equation}
     Let us consider the term ${\rm (I)}_k$.
     By the monotone convergence theorem, we can estimate
     \begin{align}
      {\rm (I)}_k 
      &=\int_{K_k^c \cap H_R}
       \widetilde{U_1 \nu}(x) \, \mu(dx) 
       = \lim_{N \to \infty} \int_{(K_N \setminus K_k) \cap H_R}
       \widetilde{U_1 \nu}(x) f(x) \, m(dx) \nonumber \\
      &= \lim_{N \to \infty} \cE_1 \big(U_1 \nu, R_1 (f \cdot 1_{K_N
       \setminus K_k} \cdot 1_{H_R})\big) 
 \nonumber  \\
      &= \lim_{N \to \infty} \int_{E} 
       \bE_{x}\left[\int_{0}^{\infty} e^{-t} (f \cdot 1_{K_N \setminus
	       K_{\ell}} \cdot 1_{H_R})(X_t) \, dt\right] \nu(dx)
       \nonumber \\
      &= \int_{E} \bE_{x} \left[\int_{0}^{\infty} e^{-t} (f \cdot
			    1_{K_{\ell}^c} \cdot 1_{H_R})(X_t) \,
			    dt\right]\nu(dx) \le
      R \cdot \nu(E) < \infty. \la{eq-rem-01}
     \end{align}
     Hence condition (Sb) holds.

     Next we consider condition (Sc).
     Put
     \[
      H_{n,a} = \{x \in E : f_n(x) \le a\} \quad \text{for}\ a > 0.
     \]
Since $\lim_{n \to \infty} \| f_n - f \|_p =0$, for any $\varepsilon > 0$,
there exists an $R_* \ge 1$
     such that 
     \[
      \int_{H_{n,R_{\ast}}^c} f_n(x)^p \, m(dx) < \varepsilon,  \quad n \in {\bN}.
     \]
Following the same argument as for (\ref{pnew3}) and (\ref{eq-rem-01}),  we know that
     for each $k \in \bN$, 
     \begin{align}
      \sup_{n \in {\mathbb N}} \int_{K_k^c} \widetilde{U_1 \nu}(x) \, \mu_n(dx)
       &= \sup_{n \in {\mathbb N}} \int_{K_k^c} \widetilde{U_1 \nu}(x) f_n(x) \, m(dx) \nonumber\\
      &= \sup_{n \in {\mathbb N}} \left(\int_{K_{\ell}^c \cap H_{n,R_{\ast}}}
			    + \int_{K_k^c \cap H_{n,R_{\ast}}^c}\right)
      \widetilde{U_1 \nu}(x) f_n(x) \, m(dx) \nonumber\\
      &\le R_{\ast} \cdot \nu(E) + \|U_1 \nu\|_{\infty} \cdot 
       \sup_{n \in {\mathbb N}} \int_{H_{n,R_{\ast}}^c} f_n(x)^p \, m(dx) \nonumber\\
      &\le R_{\ast} \cdot \nu(E) + \epsilon \|U_1 \nu\|_{\infty} <
       \infty.  \la{eq-02}
     \end{align}
     Therefore, by Fatou's lemma and \eqref{eq-02}, 
     \begin{align*}
      & \limsup_{k \to \infty} \sup_{n \in {\mathbb N}} \int_{K_k^c}
       \widetilde{U_1 \nu}(x) \, \mu_n(dx) \\
      &\le \limsup_{k \to \infty} \sup_{n \in {\mathbb N}}  \int_{E}
       \bE_x\left[\int_{0}^{\infty} e^{-t} (f_n \cdot 
	     1_{H_{n,R_{\ast}}} \cdot 1_{K_k^c})(X_t) \, dt\right]
       \nu(dx)
       + \varepsilon \|U_1 \nu\|_{\infty} \\
      &\le \int_{E} \bE_x\left[\int_{0}^{\infty} e^{-t}
			  \left\{\sup_{n\in {\mathbb N}} (f_n \cdot 1_{H_{n,R_{\ast}}})
			  \cdot \left(\limsup_{k \to \infty}
				 1_{K_k^c}\right)\right\}(X_t)
			  \, dt\right] \nu(dx)
      + \varepsilon \|U_1 \nu\|_{\infty} \\
      &= \varepsilon \|U_1 \nu\|_{\infty}. 
     \end{align*}
     Letting $\varepsilon \to 0$, we see that the condition (Sc) is satisfied. 
    \end{proof}

\begin{remark}
Assume $\lim_{n \to \infty}\| f_n - f \|_1=0$.
Then  there is a  subsequence $\{ f_{n'} \}$ such that $f_{n'} \to f$ $m$-a.e.  We put
\begin{align}
	F_k 
	:= 
	\overline{ 
		\left\{ x \in Q_k  
			\left| \, 
			\left( \sup _{n'}  f_{n'}(x) \right) \vee f(x) \le k \right. 
		\right\} 
	}, \qquad 
	k \ge 1, \label{newFF}
\end{align}
where $\{Q_k \}$ is  an increasing sequence of compact sets of $E$ such that $E=\bigcup_{k=1}^\infty Q_k$.
$ \{ F_k \}$ is an increasing sequence of compact subsets of $E$ with 
$E=\bigcup_{k =1}^\infty F_k$ $m$-a.e. 
Then (\ref{pnew1}) and (\ref{pnew2}) are available for $p=1$, $n=n'$ and $F_k$ given by (\ref{newFF}). 
Therefore (Sa) is satisfied.  The argument for (Sb) and (Sc) in the proof of  {\bf Case~II}
is also available for $p=1$, $n=n'$  and $K_k$ replaced by $F_k$ given by (\ref{newFF}).
Thus we get  a subsequence $\{ n' \}$ and an increasing subsequence $\{F_k \} $ such that $E=\bigcup_{k=1}^\infty F_k$ $m$-a,e,
and (Sa), (Sb), (Sc) are satisfied with $\{ n \}$ replaced by $\{n' \}$. 
\end{remark}

    \section{Some examples and remarks}\label{sec7}

  \subsection{A class of absolutely continuous smooth measures}

As is mentioned in Remark \ref{rem-01},  for $f\in L^p(E;m)\cap {\cal B}_+(E) \ (1\le p<\infty)$  we can 
construct approximating sequences $\{f_n\}$ and $\{F_n\}$, namely, $f \in \frA$.  
In this subsection, we introduce several examples of functions 
having some singularities so that they satisfy Definition \ref{def-A}.  To this end, let ${\sf M} = ({\mathbb P}_x,X_t)$ be the Brownian motion 
on ${\mathbb R}^d$ and $m(dx)=dx$ the Lebesgue measure on ${\mathbb R}^d \, (d\ge 2)$. We start with a simple example.

   \begin{example}      \la{ex-3b}
Assume $\beta \in {\mathbb R}$. Let 
$$ 
f(x)=|x|^{-\beta}, \quad x \in {\mathbb R}^d.
$$ 
Put $K_n=\{ x\in{\mathbb R}^d :  1/n \le |x| \le n\}$ for $n\in {\mathbb N}$.  Then we see that $\{K_n\}$ is an increasing sequence of compact sets 
so that $\bigcup_{n=1}^\infty K_n={\mathbb R}^d$ {\it q.e.}  and $f\in L^1(K_n)$ for $n\ge 1$.  So $f \in \dot{L}^1_{\sf loc}({\mathbb R}^d)$,  and  $f m =|x|^{-\beta}m \in{\mathcal S}$  by Lemma \ref{lem-S0} 
(see also \cite[Example 5.1.1]{FOT}). Of course  we know $f\in L^1_{\sf loc}({\mathbb R}^d)$ if $\beta<d$. 

On the other hand, assume $0<\beta<d$.  Then $f\in L^1_{\sf loc}({\mathbb R}^d)$ but $ f \not\in L^p(\mathbb{R}^d)$ for any 
$p \in [1,\infty]$. We now put $F_n = \{ x \in {\mathbb{R}^d} : |x| \leq n \}$ and set $ f_n = 1_{F_n} f$ for $n \ge 1$.  
Noting that $f,\ f_n  \in  L_{\sf loc}^{1}({\mathbb{R}^d}) \cap \cB _+ ({\mathbb{R}^d})$ and 
${\mathbb R}^d=\bigcup_{n=1}^\infty F_n$, we see that  $f\in \frA$ and $\{ f_n \}$ is an 
$\frA$-approximating sequence of $f$ with $\{F_n\}$ as an increasing sequence of closed sets.  Indeed, (Aa) (resp.  (Ab1) and (Ab2)) is 
satisfied for any $p_n\in [1, \infty)$ (resp. for any  $q_n \in [1,\infty)$ and $C>1$) since $f-f_n=0$ on $F_n$ 
 (resp. $f_n=0$ on ${\mathbb R}^d\setminus F_n$) hold for all $n\ge 1$.  As for (Ac), set $r_n = \frac{d+1}\beta \, (>1)$ for $n\ge 1$. 
 Then  it follows that 
\begin{align*}
C^{r_n} \|f\|_{L^{r_n}({\mathbb R}^d\setminus F_n)}^{r_n} & =C^{(d+1)/\beta} \int_{|x|\ge n} |x|^{-d-1} dx \to 0 \quad (n\to \infty)
\end{align*}
for any $C>1$.   
	\hfill \fbox{}   
\end{example}

\begin{example}

We now  put
     \[
     f(x) = \sum_{n=2}^{\infty} |x|^{-d a_n} 1_{K_n^1}(x) +
     \sum_{n=2}^{\infty} |x|^{d b_n} 1_{K_n^2}(x),
     \quad x \in \bR^d,
     \]
where 
 \begin{align*}
     K_n^1  & := \left\{x \in \bR^d :
     \left(\frac{1}{n}\right)^{1/d} < |x| \le
     \left(\frac{1}{n} + \frac{1}{n^{2a_n + 2}}\right)^{1/d}\right\},  \\
     K_n^2 &  := \left\{x \in \bR^d :
     \left(n - \frac{1}{n^{2 b_n + 2}}\right)^{1/d}
     < |x| \le n^{1/d}\right\}
  \end{align*}
and $\{a_n\},  \, \{b_n\}$ are any sequences of real numbers satisfying  $1 \le a_n, b_n$ for $n\ge 2$.
     Then we see that $f \in  L^1_{\sf loc}({\mathbb R}^d) \cap \frA$.
We first  show that $f$ belongs to $L^1_{\sf loc}({\mathbb R}^d)$. To this end, it is enough to show that $f$ is integrable on the set $\{|x|\le 1\}$.  So
\begin{align*}
\int_{\{|x|\le 1\}} |f(x)|dx & = \sum_{n=2}^\infty \int_{K^1_n} |x|^{-d a_n} dx \ 
= \sum_{n=2}^\infty   \  \int\limits_{\{(\frac1n)^{1/d} <|x|\le (\frac 1n +\frac 1{n^{2a_n+2}})^{1/d}\}} |x|^{-d a_n} dx \\
& \le \sum_{n=2}^\infty n^{a_n}  \int\limits_{\{(\frac1n)^{1/d} <|x|\le (\frac 1n +\frac 1{n^{2a_n+2}})^{1/d}\}} dx \\
& = c \sum_{n=2}^\infty n^{a_n} \cdot n^{-2a_n-2} =  c \sum_{n=2}^\infty  n^{-a_n-2} \le  c \sum_{n=2}^\infty n^{-2}<\infty
\end{align*}
since $a_n\ge 1$.  We now show that  $f\in \frA$ for  setting $\{f_n=f1_{F_n}\}$ as an $\frA$-approximating sequence of $f$, 
$p_n=q_n=1$ and  $r_n=2$, where  
     \[
     F_n := \left\{x \in \bR^d : \left(\frac{1}{n}\right)^{1/d}
     \le |x| \le n^{1/d}\right\}
     \]
for $n\ge 1$.  As in the previous example,  we see that (Aa) (resp.  (Ab1) and (Ab2)) is automatically satisfied since 
$f-f_n=0$ on $F_n$ (resp. $f_n=0$ on $F_n^c$) for any $n\ge 1$.   To show (Ac2),  take  any $C>1$ and $n\ge 2$. Then we see that 
\begin{align*}
& \!\!\! C^{r_n} \|f\|_{L^{r_n}({\mathbb R}^d \setminus F_n)}^{r_n}  \  = C^2  \|f\|_{L^2({\mathbb R}^d \setminus F_n)}^2  \\
& =C^2 \Big(
\sum_{k=n+1}^\infty \int\limits_{\{(\frac1k)^{1/d} <|x|\le (\frac 1k +\frac 1{k^{2a_k+2}})^{1/d}\}} 
  |x|^{-2d a_k} dx 
+ \sum_{k=n+1}^\infty \int\limits_{\{ (k-\frac 1{k^{2b_k+2}})^{1/d} <|x|\le k^{1/d}\}}  
 |x|^{2d b_k} dx \Big) \\
& \le C^2 c  \Big(\sum_{k=n+1}^\infty k^{2a_k} \cdot k^{-2a_k-2} + \sum_{k=n+1}^\infty k^{2b_k} \cdot k^{-2b_k-2} \Big)  \\
& = 2 C^2 c \sum_{k=n+1}^\infty k^{-2} \to 0 \quad  (n\to \infty).
\end{align*}

\hfill \fbox{}

\end{example}

\medskip

We introduce another example for which the approximating sequence of parameters varies ($\{r_n\}$,  in this case).

\begin{example} We put 
$$
f(x)=|x|^{-d} \left(\log |x| \right) ^{-2} 1_{B(0,1/e)}(x) + \sum _{n =1} ^\infty |x|^{-d + 1/\log(n+1)} 1 _{K_n} (x) , \quad x\in{\mathbb R}^d, 
$$
where $K_n = \big\{ x\in {\mathbb R} :  n \le |x| < n+1 \big\}$ for $n\ge 1$.  Then 
$f\not\in L^p({\mathbb R}^d)$ for $1\le p\le \infty$ but $f\in L^1_{\sf loc}({\mathbb R}^d)\cap \frA$.  We first show that 
$f\not\in L^p({\mathbb R}^d)$  for $1\le p<\infty$. When $p=1$, 
\begin{align*}
\|f\|_{L^1} & \ge \sum_{n=1}^\infty \int_{n \le |x| <n+1} |x|^{-d+1/\log(n+1)}dx  \\
& =  c_d \sum_{n=1}^\infty \int_n^{n+1} u^{-1-1/\log(n+1)} du \\
& \ge c_d \sum_{n=1}^\infty (n+1)^{-1}  \ \underline{ \ (n+1)^{-1/\log (n+1)} }_{\scriptsize \ {\rm bounded  \ in } \ n} =\infty.
\end{align*}
When $1<p<\infty$, 
\begin{align*}
\|f\|_{L^p}^p &  \ge  \int_{|x|<1/e} |x|^{-pd}\big( \! \log |x|\big)^{-2p}dx \ =c_d \int_0^{1/e} u^{-d(p-1)-1} (\log u)^{-2p}du \quad  (t=\log u) \\
& = c_d \int_1^\infty e^{d(p-1)t} t^{-2p}dt =\infty.
\end{align*}
As for the locally integrability  of $f$,  it is enough to see the integrability around the origin:
$$
\int_{B(0,1/e)} |x|^{-d} \big(\! \log |x|\big)^{-2} dx =c_d \int_0^{1/e} u^{-1} (\log u)^{-2}du =c_d \int_1^\infty t^{-2} dt <\infty \quad (t=\log u).
$$
We now show $f\in \frA$.  To this end, put $F_n=\{ x\in {\mathbb R}^d : |x| \le n\}$, $p_n=q_n=1$,  $r_n=\log (n+1)$ and set
$f_n=f 1_{F_n}$ for $n \ge1$.  As in the previous examples, we see that $\|f-f_n\|_{L^{p_n}(F_n)}=\|f_n\|_{L^{q_n}(F_n^c)}=0$ for any $n\ge 1$, hence 
(Aa) and (Ab2) hold.   We finally show (Ac1).  Let $C>1$. Then for any $n \ge 2$, 
\begin{align*} 
\sum_{k=n}^\infty \|f\|_{L^{r_k}(F_{k+1}\setminus F_k)} & = 
\sum_{k=n}^\infty \Big( \int_{k \le |x|< k+1} |x|^{-d\log (k+1) +1} dx \Big)^{1/\log(k+1)}   \\
& = \sum_{k=n}^\infty \Big( c_d \int_k^{k+1} u^{-d(\log(k+1)-1)} du \Big)^{1/\log(k+1)}   \\
& \le  \sum_{k=n}^\infty k^{-d} \underline{ \ (c_dk^d)^{1/\log(k+1)} }_{ \ \scriptsize \ {\rm bounded \ in} \ k}  \ \longrightarrow  0 \quad (n\to \infty)
\end{align*}
because $d\ge 2$. \hfill \fbox{}

\end{example}

We noted in Remark \ref{rem-01} that the condition (Ab1) (resp.  (Ab2)) is not necessarily stronger 
than (Ab2) (resp. (Ab1)) and so are the conditions (Ac1) and (Ac2).  
The following example shows such a situation indeed happens for the conditions (Ac1) and (Ac2).

\begin{example}
\begin{itemize}
\item[(i)] Set $f(x)=|x|^{-1} 1_{B(0,1)^c}(x), \ x\in \real^d$. Then we find that 
the condition (Ac2) holds but not (Ac1) for $r_n=d+1$ and $F_n=\{|x|\le n\}$. In fact, 
similar to \eqref{ex-3b}, we see that, for any $C>1$ and $n\ge 1$, 
$$
C^{r_n} \|f\|_{L^{r_n}(F_n^c)}^{r_n}  = C^{d+1} \int_{|x|\ge n} |x|^{-d-1}dx 
=  c_d C^{d+1} \int_n^\infty u^{-2}du \to 0 \quad (n\to\infty)
$$
holds, and then (Ac2) is satisfied.  As a result, it follows $f\in{\mathfrak A}$.
But we see (Ac1) fails since, for any $n\ge 1$, 
\begin{align*}
\sum_{k=n}^\infty \|f\|_{L^{r_k}(F_{k+1}\setminus F_k)} 
&  = \sum_{k=n}^\infty  \Big(\int_{k<|x|\le k+1} |x|^{-d-1}dx \Big)^{1/(d+1)}  \\
&  = \sum_{k=n}^\infty  \Big(\frac{c_d}{k(k+1)}\Big)^{1/(d+1)} \ge c_d^{1/(d+1)} \sum_{k=n}^\infty 
(k+1)^{-2/(d+1)} =\infty
\end{align*}
and $m(F_{k+1}\setminus F_k)={\sf vol}(\{ k<|x|\le k+1\}) \ge  c_d k^{d-1} >1$ for $k\ge 1$.

\item[(ii)] Put $\ds f(x)=\sum_{k=1}^\infty |x|^d 1_{K_{2k-1}}(x) + \sum_{k=1}^\infty |x|^{-d} 1_{K_{2n}}(x), \  
x\in \real^d$, where 
$$\left\{
\begin{array}{rl}
K_{2k-1} & := \Big\{x\in \real^d : \  (2k-1)^{1/d} < |x| \le  \big(2k-1 +(2k-1)^{-3}\big)^{1/d} \Big\},  \\
& \vspace*{-6pt} \\
K_{2k} & := \Big\{x\in \real^d : \  (2k)^{1/d} < |x| \le  \big(2k +(2k)^{-\delta \log(2k+2)}\big)^{1/d} \Big\} \\
\end{array}
\right.
$$
for $k\ge 1$ and for some $\delta>1$. Then, for  $F_n=\{|x|\le n^{1/d} \}$ and 
$r_n= \left\{ 
\begin{array}{ll} 
\log(n+1)  &  {\rm if} \  n\   {\rm is \ even}, \\ 
 1  &  {\rm if}  \  n \ {\rm  is \  odd},  \\
 \end{array} \right. 
 $
	we see that  the condition (Ac1) holds but (Ac2) fails.  
	Take any $C>1$ and any even number $n\ge 1$, namely $n=2m$ for any $m \ge 1$,  then we see that 
\begin{align*}
C^{r_n} \|f\|_{L^{r_n}(F_n^c)} 
& =C^{\log (2m+1)} \int_{|x|>(2m)^{1/d}} |f(x)|^{\log(2m+2)}dx \\
& \ge C^{\log (2m+1)} \sum_{k\ge 2m \atop k=2\ell-1: {\rm odd}}  \ \ 
\int\limits_{ k^{1/d} < |x| \le  (k+1/k^3)^{1/d}} |x|^{d\log(2m+2)}dx \\
& \ge c_d C^{\log (2m+1)} \sum_{\ell=m+1}^\infty (2\ell-1)^{\log(2m+2)-3} =\infty.
\end{align*}
Thus this means (Ac2) fails. On the other hand, for $n\ge 2$, 
\begin{align*}
\sum_{k=n}^\infty \|f\|_{L^{r_k}(F_{k+1}\setminus F_k)}  
& =\sum_{k=n}^\infty  \Big( \int\limits_{k^{1/d} < |x| \le (k+1)^{1/d}}  \hspace*{-10pt} |f(x)|^{r_k}dx \Big)^{1/r_k} \\
& = \sum_{k\ge n \atop k=2m}^\infty  
\Big( \int\limits_{(2m)^{1/d} < |x| \le (2m+1)^{1/d}}    |f(x)|^{r_{2m}}dx \Big)^{1/r_{2m}} \\
& \qquad 
+ \sum_{k\ge n \atop k=2m-1}^\infty  
\Big( \int\limits_{(2m-1)^{1/d} < |x| \le (2m)^{1/d}}    |f(x)|^{r_{2m-1}}dx \Big)^{1/r_{2m-1}} \\
& = \sum_{k\ge n \atop k=2m}^\infty  
\Big( \int\limits_{K_{2m}}    |x|^{-d\log(2m+2)}dx \Big)^{1/\log(2m+2)} + 
 \sum_{k\ge n \atop k=2m-1}^\infty  \Big(\int_{K_{2m-1}} |x|^d dx \Big) \\
 & \le c_d \sum_{m=[n/2]}^\infty (2m)^{-\delta-1} 
 + c_d \sum_{m=[(n+1)/2]}^\infty (2m+1) \cdot (2m-1)^{-3}
\end{align*}
goes to $0$ when $n\to\infty$, and hence, (Ac1) holds. We also see that $f$ belongs to ${\mathfrak A}$.
\hfill \fbox{}
\end{itemize}

\end{example}

\subsection{Remarks}

In the following, we denote by ${\sf A}^{\mu}_t$ the {\sf PCAF} in the Revuz correspondence for 
$\mu \in {\cS}$. If $\mu(dx)= f(x) m(dx)$ for some  $f \in \dot{L}_{\sf loc}^1(E;m) \cap \cB_+(E)$, 
then $\mu \in {\cS} $ and ${\sf A}^\mu_t = \int_0^t f(X_s)ds$ by Lemma~\ref{lem-S0}. 
In this subsection, we present two cases in which the corresponding {\sf PCAF} ${\sf A}^{\mu}_t$ can be 
concretely expressed  when $\mu \in {\cS}$ is not necessarily absolutely continuous with respect to $m$. 

\subsubsection{One-dimensional diffusion processes}

Let ${\sf X} =[\Omega^{\sf X}, X_t, {\mathbb P}^{\sf X}_x, \zeta^{\sf X}]$  be a one-dimensional 
diffusion process ({\sf  ODDP} for short)  on an open interval  $I=(l_1,l_2)$ with no killing inside $I$. 
It is well known that ${\sf X}$ is characterized by a scale function $s^{\sf X}$ and a speed measure 
$m^{\sf X}$, and a boundary condition (reflecting or absorbing) at $l_i$ if it is regular.  $s^{\sf X}$ 
is a strictly increasing continuous function on $I$ and $m^{\sf X}$ is a positive Radon measure on $I$
with $\mbox{supp} [m^{\sf X}]=I$.  $m^{\sf X}$ is identified with strictly increasing right continuous 
function induced by $m^{\sf X}$.

We consider the following symmetric
bilinear form $({\cE}^{\sf X},{\cF}^{\sf X})$:
\begin{align*}
	&{\cE}^{\sf X}(u,v) =
	\int_I \frac{du}{ds^{\sf X}}\,\frac{dv}{ds^{\sf X}}\,ds^{\sf X},\\
	&{\cF}^{\sf X}  = \Big\{
	u \in L^2(I^*; m^{\sf X}) : u\ \mbox{satisfies (\ref{tu1}), (\ref{tu2}) and (\ref{tu3})} \Big\}, 
\end{align*}
where $I^*$ is the interval extended by adding $l_i$ to $I$
if $l_i$ is regular (i.e.  $| s^{\sf X}(l_i)| + | m^{\sf X}(l_i)| < \infty$)  with reflecting boundary condition,
and $m^{\sf X}$ is extended to $I^*$ so that $m^{\sf X}(I^* \setminus I)=0$.
\begin{align}\label{tu1}
&\ u(x) \ \mbox{is absolutely continuous with respect to} \ s^{\sf X}(x) \ \mbox{on}\  I,\\
& \label{tu2}
\ \frac{du}{ds^{\sf X}} \in L^2(I; s^{\sf X}),\qquad
u(b)-u(a)= \int_{(a,b)} \frac{du}{ds^{\sf X}}\, ds^{\sf X}\quad (a,b \in I),
\\
&
\lim_{x \to l_i}u(x)=0\ \mbox{if\ $|s^{\sf X}(l_i)| <\infty$ but $l_i$ is not regular with reflecting boundary condition} \  (i=1,2).  \label{tu3}
\end{align}
Then  $({\cE}^{\sf X}, {\cF}^{\sf X})$ is a regular, strongly local, irreducible Dirichlet form on 
$L^2(I^*;m^{\sf X})$ corresponding to the {\sf ODDP} ${\rm X}$, and each one point of $I^*$ has a 
positive capacity relative to  $({\cE}^{\sf X}, {\cF}^{\sf X})$ (\cite{CF}, \cite{F2010}, \cite{F2014}).
We denote the corresponding items for $({\cE}^{\sf X}, {\cF}^{\sf X})$ by $\mbox{Cap}^{\sf X}$, 
${\cS}^{\sf X}$,  ${\cS}^{\sf X}_0$,  ${\cS}^{\sf X}_{00}$, etc.  
In the same way as in \cite[Example~2.1.2]{FOT},  we obtain the following.
\begin{align} \label{cap}
\mbox{Cap}^{\sf X}( \{ y \})= 1/g^{\sf X}_1(y,y)\  (>0),\ y \in I^*,
\end{align}
where
$g^{\sf X}_1(x, y)$ is the resolvent kernel corresponding to $({\cE}^{\sf X}, {\cF}^{\sf X})$.
We also note that the set $\big\{ u(s^{\sf X}(\cdot )): u \in C^1_c(J^*)\big\}$
is a core of $({\cE}^{\sf X},{\cF}^{\sf X})$, where $J^*=s^{\sf X}(I^*)$ and $C^1_c(A)$ is the set of all 
$C^1$-functions on $A$ with compact support.

Let $\mu$ be  a positive Borel measure  on $I^*$.  Noting (\ref{cap}) and following the same argument 
as in  \cite[Example~2.1.2]{FOT},  we see the following.
\begin{align} \label{S1}
\mu \in {\cS}^{\sf X} \quad \mbox{if and only if}\quad \mbox{$\mu$ is a positive Radon measure on $I^*$}.
\end{align}
By virtue of 
\cite[Exercise 4.2.2]{FOT}, we see that 
\begin{align}\label{S2}
\mu \in {\cS}^{\sf X}_0 \quad \mbox{if and only if}\quad 
\iint_{I^* \times I^*} g^{\sf X}_1(x,y) \mu(dx) \mu(dy) <\infty,
\end{align}
and, in this case,
\[
R^{\sf X}_1 \mu(x) : = \int_{I^*} g^{\sf X}_1(x,y) \mu(dy)
\]
is a quasi continuous and 1-excessive version of $U^{\sf X}_1 \mu$.
By means of (\ref{S1}) and (\ref{S2}), we get $  {\cS}^{\sf X}_{00}={\cS}^{\sf X}_0= {\cS}^{\sf X}$ 
if $|s^{\sf X}(l_i)|+|m^{\sf X}(l_i)|<\infty$ and $l_i$ is reflecting for $i=1$ and 2. 

We assume that {\sf X} is conservative, that is, $\zeta^{\sf X}=\infty$. 
Let $\mu \in {\cS}^{\rm X}$.  Employing \cite[Theorem 5.1.4]{FOT},  we have
\begin{align}
\label{AL}
{\sf A}^\mu_t = \int_{I^*}  \ell^{\sf X}(t, x) \mu(d x),
\end{align}
where $\ell^{\sf X}(t, x) $ is the local time of {\sf X} and it is continuous with respect to $(t, x) \in [0,\infty) \times I^*$
and satisfies  $ \int_0^t 1_A(X_u)\, d u = \int_A \ell^{\sf X} (t , x)\,m^{\sf X}(dx), \ t>0$,
for every measurable set $A \subset I$,
where $1_A$ is the indicator for a set $A$ (\cite{IM}).  The relation (\ref{AL}) 
 for one-dimensional Brownian motion is obtained in  \cite[Example~5.1.1]{FOT}.
Let $\mu_n,\ \mu \in {\cS}^{\sf X}$ and assume $\mu_n \to \mu$ vaguely as $n \to \infty$.
By virtue of (\ref{AL}), we obtain the following.
\begin{align} \label{ALn}
{\mathbb P}^{\sf X}_x \left( \lim_{n \to \infty} {\sf A}^{\mu_n}_t = {\sf A}^{\mu}_t\ \mbox{locally  uniformly in $t$ 
 on} \ [0,\infty) \right)=1, \quad x \in I^*.
\end{align}


\subsubsection{Skew product diffusion processes}

Let ${\sf X}=[\Omega^{\sf X}, X_t, {\mathbb P}^{\sf  X}_x, \zeta^{\sf X}]$ be the {\sf ODDP} as above,
where $ 0=l_1 <l_2 \leq \infty$ and $ l_1=0$ is entrance in the sense of Feller, that is, 
$\int_{(0, c)} m^{\sf X} (dx) \int_{(x,c)} s^{\sf X} (dy) <\infty$ and 
$\int_{(0, c)} s^{\sf X} (dx) \int_{(x,c)} m^{\sf X} (dy) =\infty$ for $ 0<c <l_2$.
Let
$ \Theta =[\Omega^\Theta, \Theta_t, {\mathbb P}^{\Theta}_\theta]$ be
the spherical Brownian motion on $ S^{d-1} \subset {\bR}^d$ with 
generator $ \frac{1}{2} \Delta$, $\Delta$ being the spherical Laplacian on 
$ S^{d-1}$.
The corresponding Dirichlet form $({\cE}^{\Theta},{\cF}^{\Theta})$
on $L^2(S^{d-1}, m^{\Theta})$
is obtained by Fukushima and Oshima \cite{FO},
where $m^{\Theta}$ is the spherical measure on $S^{d-1}$. 
Let $  \nu$ be a Radon measure on $I$ and assume that
$\mbox{supp[$ \nu$]}=I$.
Furthermore we assume that 
$\int_{(0,c)} | s^{\sf X} (x)| \nu(dx)=\infty$, and 
 $ \int_{(c, l_2)} \nu (dx) <\infty$
whenever $s^{\sf X}(l_2) + m^{\sf X}(l_2)<\infty$.
We set 
\begin{align}\label{eq217}
	{\bf f}(t)=
	\left\{
	\begin{array}{ll}
\ds	\int_I \ell^{\sf X}(t,x)\, \nu(dx), & t < \zeta^{\sf X},\\
	\infty, & t \geq \zeta^{\sf X}.	
	\end{array}
	\right.
\end{align}
Put  $ \Lambda := I^* \times S^{d-1}$ and
let ${\Xi} =[\Xi_t=(X_t, \Theta_{{\bf f}(t)}),
{\mathbb P}^{\Xi}_{\xi}=
{\mathbb P}^{\sf X}_x \otimes {\mathbb P}^{\Theta}_\theta,\ \xi =(x,\theta) \in \Lambda, \zeta^\Xi]$ 
be the skew product  of the {\sf ODDP} {\sf X} and the spherical Brownian motion ${\Theta}$ with 
respect to the {\sf PCAF} $\{{\bf f}(t)\}$.
Let us consider the following bilinear form.

\begin{align}
{\cE}^{\Xi}(f,g)=&\int_{S^{d-1}} {\cE}^{\rm X}( f(\cdot, \theta),
g(\cdot, \theta))\, m^{\Theta}(d \theta)
+
\int_{I} {\cE}^{\Theta}(f(x,\cdot),g(x,\cdot))\,  \nu(dx),
\label{eq219}
\end{align}
for $f,g \in {\cC}^{\Xi}$,  where ${\cC}^{\Xi}$
is the set of all linear combinations of $u^{(1)}(s^{\rm X}(x)) u^{(2)}(\theta)$ with 
$u^{(1)} \in C_0^1(J^*)$  and $u^{(2)} \in C^\infty(S^{d-1})$,
where $J^*=s^{\sf X}(I^*)$.
We put $m^\Xi =m^{\sf X} \otimes m^{\Theta}$.
By using some results from  \cite{FO} and \cite{O}, we see that
the form $({\cE}^{\Xi},{\cC}^{\Xi})$ 
is closable on $ L^2( \Lambda;m^\Xi)$, and the closure $({\cE}^{\Xi},{\cF}^{\Xi})$ is a regular irreducible Dirichlet 
form and it is corresponding to the skew product ${\Xi}$.
We denote the corresponding items for $({\cE}^{\Xi}, {\cF}^{\Xi})$ by ${\cS}^{\Xi}$,  
${\cS}^{\Xi}_0$,  ${\cS}^{\Xi}_{00}$, etc. 
We show the following.

\begin{proposition}\label{prS0}
Let $\mu^{\sf X}$ be a positive Radon measure on $I^*$.
Put 
$\mu^\Xi(dx, d\theta) =\mu^{\sf X} (dx) m_o(\theta) m^{\Theta}(d\theta)$, where
$m_o$ is a bounded measurable function on $S^{d-1}$. Then the following hold$:$

\begin{itemize}
\item[\rm (1)]  $\mu^\Xi \in {\cS}^{\Xi}$.

In particular, if ${\sf X}$ is conservative and $m_o=1$, then
\begin{align} \label{AS}
{\sf A}^{\mu^\Xi}_t= \int_{I^*} \ell^{\sf X}(t, x) \mu^{\sf X}( dx), \qquad {\mathbb P}^{\Xi}_{\xi} 
\text{-a.e.}
\quad
\xi \in \Lambda.
\end{align}

\item[\rm (2)] If  $\mu^{\sf X} \in {\cS}_0^{\sf X}$  \ (resp.~${\cS}_{00}^{\sf X}$), then 
$\mu^\Xi \in {\cS}_0^{\Xi}$  \ (resp.~${\cS}_{00}^{\Xi}$).
\end{itemize}
\end{proposition}

\medskip
We show this by means of the following lemma which 
can be easily shown from 
the fact that ${\cC}^\Xi$ is a core of $({\cE}^{\Xi}, {\cF}^{\Xi})$, 
so we will omit its proof.

\begin{lemma}\label{lemS0}

Let $V$ be a function on $\Lambda $ and put $v(x)= \int_{S^{d-1}} V(x, \theta) m_o(\theta) m^{\Theta}(d \theta)$.
\begin{itemize}
\item[\rm (1)]  If $V \in {\cF}^\Xi $ \ (resp. $ C_0(\Lambda) \cap {\cF}^\Xi $), then 
$ v \in  {\cF}^{\sf X} $ \ (resp. $ C_0( I^* ) \cap {\cF}^{\sf X}$), and
\begin{align*}
 {\cE}^{\sf X}(v,v)  \leq A_{d-1} \| m_o\|^2_{L^\infty(S^{d-1})}   {\cE}^\Xi(V,V),\qquad 
 & \| v\|^2_{L^2(I^* , m^{\rm X})}\leq
 A_{d-1}  \| m_o\|^2_{L^\infty(S^{d-1})}  \| V \|^2_{L^2(\Lambda, m^\Xi)}.
\end{align*}
\item[\rm (2)]  Let $u \in {\cF}^{\sf X}$. Then $U(x, \theta) :=u(x) \in {\cF}^\Xi$ and
${\cE}^\Xi_1 (U, U) =A_{d-1} {\cE}_1^{\sf X}(u, u)$.

In particular, if $V \in {\cF}^\Xi $ and $m_o=1$, then
$ {\cE}^\Xi_1 (U, V) = {\cE}_1^{\sf X}(u,v)$.
\end{itemize}
\end{lemma}

 \begin{proof}[Proof of Proposition \ref{prS0}]
(step 1)\  Let $\mu^{\sf X} \in {\cS}^{\sf X}_0$ and put 
$\mu^\Xi(dx, d\theta) =\mu^{\sf X} (dx) m_o(\theta) m^{\Theta}(d\theta)$. 
By means of Lemma~\ref{lemS0}, we get 
\[
\int_{\Lambda} | V(x, \theta)| \mu^\Xi(dx,  d\theta) \leq C \sqrt{ {\cE}^\Xi_1(V,V) },\quad 
V \in C_0( \Lambda) \cap {\cF}^\Xi,
\]
for some positive constant $C$,
which implies  $\mu^\Xi \in {\cS}_0^\Xi$.
Assume that {\sf X} is conservative and $m_o=1$.
We put $C_t (\omega, \omega')= \int_{I^*} \ell^{\sf X}(t, x; \omega) \mu^{\sf X}(dx)$ 
for $\omega \in \Omega^{\rm X},\ \omega' \in \Omega^{\rm \Theta}$.
We show 
\begin{align}\label{EQ1}
{\mathbb P}^\Xi_{\xi } ( {\sf A}^{\mu^\Xi}_t =C_t)=1,\ \xi  \in  \Lambda.
\end{align}
From now on, 
we denote the expectation with respect to $\bP_{\ast}^{\Xi}$
(resp. $\bP_{\ast}^{\sf X}$)
by $E^{\bP_{\ast}^{\Xi}}$ (resp. $E^{\bP_{\ast}^{\sf X}}$). 
Let $\xi =(x,\theta) \in \Lambda$.
Noting  \cite[Sect.~5.4, (3)]{IM},
\begin{align*}
E^{{\mathbb P}^\Xi_{\xi}}\left[ \int_0^\infty e^{-t} dC_t \right] & =
 \int_0^\infty e^{-t} E^{{\mathbb P}^\Xi_{\xi}}\left[C_t \right] dt\\
\nonumber
&=
 \int_0^\infty e^{-t} E^{{\mathbb P}^{\sf X}_{x}}\left[C_t \right] dt
=
\int_0^\infty e^{-t} dt \int_{I^*}  E^{{\mathbb P}^{\sf X}_{x}} [ \ell^{\sf X}(t, y) ] \mu^{\sf X}(d y)\\
\nonumber
&=\int_0^\infty e^{-t} dt \int_{I^*}  \mu^{\sf X}(d y) \int_0^t p^{\sf X}(s,x,y)ds\\
&= \int_{I^*}  \mu^{\sf X}(d y) \int_0^\infty e^{-s}  p^{\sf X}(s, x,y)ds\\
&= \int_{I^*} g^{\sf X}_1(x,y) \mu^{\sf X}(d y) =(R^{\sf X}_1 \mu^{\sf X})(x)
=(U^{\sf X}_1 \mu^{\sf X})(x)
=U^\Xi_1 \mu^\Xi(\xi),
\end{align*}
where the last equality follows from Lemma~\ref{lemS0}.
Combining this with \cite[Theorem~5.1.2]{FOT}, we get (\ref{EQ1}).

\smallskip
(step 2)\ Let $\mu^{\sf X} \in {\cS}_{00}^{\sf X}$.
By means of (step 1), $\mu^\Xi \in {\cS}^\Xi_0$.
Then 
$$
\mu^\Xi(\Lambda)\leq \mu^{\sf X}(I^*) \| m_o\|_{L^\infty(S^{d-1};m^\Theta)} A_{d-1} <\infty.
$$ 
Note $U^\Xi_1 \mu^\Xi \leq \| m_o \|_{L^\infty( \Lambda; m^\Xi)} U^\Xi_1 \mu^\Xi_o$\ $m^\Xi$-a.e., 
where $\mu^\Xi_o(dx d\theta) = \mu^{\sf X} (dx) m^\Theta( d \theta)$.
As in (step 1),
 $(U^{\sf X}_1 \mu^{\sf X})(x) = (U^\Xi_1 \mu_o^\Xi)(x,\theta)$
and hence 
$  \|  U^\Xi_1 \mu^\Xi \|_{L^\infty (\Lambda;m^\Xi)} \leq  
\| m_o \|_{L^\infty( \Lambda;m^\Xi)}  \|   U^{\sf X}_1 \mu^{\sf X} \|_{ L^\infty( I^*;m^{\sf X})}  < \infty$.
Therefore $ \mu^\Xi \in {\cS}^\Xi_{00} $.

\smallskip
(step 3)\ Let $\mu^{\sf X} \in {\cS}^{\sf X}$.
By means of \cite[Theorem 2.2.4]{FOT}, there exists a generalized compact nest
$ \{ F_n\}$ satisfying $ \mu^{\sf X}( I^* \setminus \bigcup_{n=1}^\infty F_n) =0$ 
and $\mu^{\sf X}_n := 1_{F_n} \mu^{\sf X} \in {\cS}^{\sf X}_{00}$ for each $n$.  
Put $F^\Xi_n = F_n \times S^{d-1}$.  Then $\{ F^\Xi_n \}$ is a generalized compact nest and 
$ \mu^\Xi ( \Lambda \setminus \bigcup_{n=1}^\infty F^\Xi_n )=0$.
By (step 2),  $\mu^\Xi_n (dx d \theta) := 1_{F^\Xi_n}(x,\theta) \mu^\Xi (dx d\theta) 
= 1_{F_n}(x)\mu^{\sf X}(dx) m_o(\theta) \mu^\Theta( d \theta) \in {\cS}^\Xi_{00}$.
Therefore $\mu^\Xi \in {\cS}^\Xi$  by means of \cite[Theorem 2.2.4]{FOT}.

Assume that {\sf X} is conservative and $m_0=1$. We show (\ref{AS}).
We first note that, by (step 1) and (step 2),
\[
{\sf A}^{\mu^\Xi_n}_t = \int_{F_n} \ell^{\sf X} (t,x)\, \mu^{\sf X}(dx) \to C_t := \int_{I^*} \ell^{\sf X}(t,x) \mu^{\sf X}(dx),
\]
uniformly in $t \in [0, T]$, $T \in (0, \infty)$ as $n \to \infty$\ ${\mathbb P}^\Xi_\xi$-a.s.,  \ $\xi \in \Lambda$.
Let $f$ be a bounded continuous function on $\Lambda$ and $ h \in {\cB}^+(\Lambda)$. 
\begin{align*}
\int_\Lambda h(\xi) m^\Xi(d \xi)
E^{\bP_{\xi}^{\Xi}}\left[ \int_0^t f( \Xi_s) d C_s  \right]
&=\lim_{n \to \infty}
 \int_\Lambda h(\xi) m^\Xi(d \xi)
E^{\bP_{\xi}^{\Xi}}\left[ \int_0^t f( \Xi_s) d A_s^{\mu^{\Xi}_n}  \right]
\\
&=\lim_{n \to \infty}
\int_0^t \langle f \cdot \mu^\Xi_n , p^\Xi_s h \rangle ds 
=
\int_0^t \langle f \cdot \mu^\Xi , p^\Xi_s h \rangle ds.
\end{align*}
This shows (\ref{AS}).

\end{proof}


\medskip
By means of (\ref{AS}), we get the following.

\begin{proposition}\label{prS0n}
Let $\mu^{\sf X}_n,\ \mu^{\sf X} $ be  positive Radon measures on $I^*$ and put
$\mu^\Xi_n(dx, d\theta) =\mu^{\sf X}_n (dx)  m^{\Theta}(d\theta)$, 
$\mu^\Xi(dx, d\theta) =\mu^{\sf X} (dx)  m^{\Theta}(d\theta)$.
Assume that 
 ${\sf X}$ is conservative and 
 $\mu^{\sf X}_n \to \mu^{\sf X} $ vaguely as $ n \to \infty$.
Then
\begin{align} \label{ASn}
{\mathbb P}^{\Xi}_\xi \left( \lim_{n \to \infty} {\sf A}^{\mu^{\Xi}_n}_t = {\sf A}^{\mu^{\Xi}}_t\ \mbox{locally 
uniformly in $t$ on}\ [0,\infty) \right)=1, \quad \xi \in \Lambda.
\end{align}

\end{proposition}



\begin{thebibliography}{0}
\bibitem{AK} S. Andres and N. Kajino,
Continuity and estimates of the Liouville heat kernel with applications to spectral dimensions,  
{\it Probab. Th. Rel. Fields}, {\bf 166} (2016), 713--752.

\bibitem{CF} Z.-Q. Chen and M. Fukushima,
{\it Symmetric Markov Processes, Time Change, and Boundary Theory}, 
Princeton University Press, 2011.


\bibitem{CTU} Z.-Q. Chen, M. Tomisaki  and T. Uemura,  
On convergences of perturbed Dirichlet forms,  
{\it in preparation},  (2024).


\bibitem{F2010} M. Fukushima, 
From one dimensional diffusions to symmetric Markov processes,  
{\it Stochastic Processes Appl.},  {\bf 120}  (2010), 590--604, 
(Special issue {\it A tribute to Kiyosi It\^{o}}).

\bibitem{F2014} M. Fukushima, 
On general boundary conditions for one-dimensional diffusions with symmetry,
{\it J.~Math.~Soc.~Japan}, 
 {\bf 66} (2014), 289--316.

\bibitem{FO} M. Fukushima and Y. Oshima, 
On the skew product of symmetric diffusion processes, 
{\it Forum Math,}  {\bf 1} (1989), 103--142.

\bibitem{FOT} M. Fukushima, Y. Oshima and M. Takeda,
{\it Dirichlet Forms and Markov Processes}. 2nd Ed.
Walter de Gruyter, 2011.

\bibitem{GRV} C. Garban, R. Rhodes and V. Vargas,
{\it Liouville Brownian motion}, {\it Ann. Probab.}, 
{\bf 44} (2016), 3076--3110.

\bibitem{IM} 
K.~It\^{o} and H.~P.~McKean,  {\it Diffusion Processes and their Sample Paths}, 
Springer-Verlag, New York, 1974.

\bibitem{Land} N. S. Landkof, 
{\it Foundations of Modern Potential Theory}, 
Springer, New York-Heidelberg, 1972.

\bibitem{O} H. Okura, 
A new approach to the skew product of symmetric Markov processes, 
{\it Memoirs of Faculty of Engineering and Design, Kyoto Institute of Technology},  
{\bf 46} (1998).

\bibitem{O} T. Ooi, 
Convergence of processes time-changed by Gaussian multiplicative chaos, 
arXiv:mathPR/2305.00734, (2023).

 \bibitem{Sh} M.J.  Sharpe,  {\it General Theory of Markov Processes},
	 Academic Press, 1988.
\end{thebibliography}
\end{document}